\documentclass[12pt, leqno]{article}
\usepackage{amsmath,amsfonts,amsthm,amscd,amssymb,verbatim} 
\usepackage[all]{xy}

\numberwithin{equation}{section}

\theoremstyle{plain}
\newtheorem{thm}{Theorem}[section]
\newtheorem{defn}[thm]{Definition}
\newtheorem{prop}[thm]{Proposition}
\newtheorem{lem}[thm]{Lemma}
\newtheorem{cor}[thm]{Corollary}

\newtheorem{propdefn}[thm]{Proposition-Definition}
\newtheorem{thmdefn}[thm]{Theorem-Definition}

\theoremstyle{definition}
\newtheorem{rem}[thm]{Remark}
 
\newtheorem{hypo}[thm]{Hypothesis}

\newcommand{\beast}{\begin{eqnarray*}}
\newcommand{\east}{\end{eqnarray*}}

\newcommand{\N}{{\Bbb N}}
\newcommand{\Z}{{\Bbb Z}}
\newcommand{\Q}{{\Bbb Q}} 
\newcommand{\R}{{\Bbb R}}

\newcommand{\E}{{\bold E}}
\newcommand{\F}{{\bold F}}

\newcommand{\Af}{{\Bbb A}}

\newcommand{\Spec}{{\mathrm{Spec}}\,}

\newcommand{\Spf}{{\mathrm{Spf}}\,}
\newcommand{\Spm}{{\mathrm{Spm}}\,}

\newcommand{\lra}{\longrightarrow}

\newcommand{\hra}{\hookrightarrow}

\newcommand{\lla}{\longleftarrow}

\newcommand{\ti}[1]{\widetilde{#1}}

\newcommand{\ul}[1]{\underline{#1}}
\newcommand{\ol}[1]{\overline{#1}}
\newcommand{\os}{\overset}

\newcommand{\conv}{{\mathrm{conv}}}

\newcommand{\im}{{\mathrm{Im}}}

\newcommand{\id}{{\mathrm{id}}}

\newcommand{\End}{{\mathrm{End}}}

\newcommand{\res}{{\mathrm{res}}}

\newcommand{\pr}{{\mathrm{pr}}}
\newcommand{\dlog}{{\mathrm{dlog}}\,}

\newcommand{\an}{{\mathrm{an}}}

\newcommand{\Frac}{{\mathrm{Frac}}\,}

\newcommand{\cE}{{\cal E}}

\newcommand{\cO}{{\cal O}}
\newcommand{\cP}{{\cal P}}
\newcommand{\cQ}{{\cal Q}}

\newcommand{\cU}{{\cal U}}
\newcommand{\cV}{{\cal V}}
\newcommand{\cX}{{\cal X}}
\newcommand{\cY}{{\cal Y}}
\newcommand{\cZ}{{\cal Z}}

\newcommand{\fD}{{\frak D}}

\newcommand{\fU}{{\frak U}}
\newcommand{\fX}{{\mathfrak{X}}}
\newcommand{\fY}{{\mathfrak{Y}}}

\newcommand{\NID}{{\mathrm{(NID)}}}
\newcommand{\NLD}{{\mathrm{(NLD)}}}

\newcommand{\sing}{{\mathrm{sing}}}

\renewcommand{\sp}{{\mathrm{sp}}}

\newcommand{\wh}{\widehat}

\renewcommand{\res}{{\mathrm{res}}}
\renewcommand{\d}{\dagger}

\newcommand{\lam}{\lambda}

\newcommand{\pa}{\partial}
\newcommand{\Mat}{{\mathrm{Mat}}}
\newcommand{\e}{{\bold e}}
\newcommand{\f}{{\bold f}}
\renewcommand{\v}{{\bold v}}
\newcommand{\Exp}{{\mathrm{Exp}}}
\newcommand{\bP}{{\bold P}}

\renewcommand{\End}{{\mathrm{End}}}

\newcommand{\dsum}{\displaystyle\sum}
\newcommand{\dlim}{\displaystyle\lim}

\begin{document}
\title{Cut-by-curves criterion for the 
log extendability of overconvergent isocrystals}
\author{Atsushi Shiho
\footnote{
Graduate School of Mathematical Sciences, 
University of Tokyo, 3-8-1 Komaba, Meguro-ku, Tokyo 153-8914, JAPAN. 
E-mail address: shiho@ms.u-tokyo.ac.jp \, 
Mathematics Subject Classification (2000): 12H25.}}
\date{}
\maketitle

\begin{abstract}
In this paper, we prove a `cut-by-curves criterion' for 
an overconvergent isocrystal on a smooth variety 
over a field of characteristic $p>0$ to extend 
logarithmically to its smooth compactification 
whose complement is a strict normal crossing divisor, under 
certain assumption. 
This is a $p$-adic analogue of a version of cut-by-curves criterion for 
regular singuarity of an integrable connection on a smooth variety 
over a field of characteristic $0$. In the course of the proof, 
we also prove a kind of cut-by-curves criteria on solvability, 
highest ramification break and exponent of $\nabla$-modules. 
\end{abstract}

\tableofcontents

\section*{Introduction}

Let $K$ be a complete discrete valuation field of mixed characteristic 
$(0,p)$ with ring of integers $O_K$ and 
residue field $k$, and let $X \hra \ol{X}$ be an open 
immersion of smooth $k$-varieties such that $Z = \ol{X}-X = 
\bigcup_{i=1}^rZ_i$ is a 
simple normal crossing divisor. Denote the log structure on $\ol{X}$ 
associated to $Z$ by $M_{\ol{X}}$. 
Let $\Sigma = \prod_{i=1}^r\Sigma_i$ be a subset of 
$\Z_p^r$ which satisfies the conditions 
$\NID$ (non-integer difference) and 
$\NLD$ ($p$-adically non-Liouville difference). 
(For precise definition, see \cite[1.8]{sigma} or 
Subsection 1.5 in this paper.) 
In the previous paper \cite{sigma}, we introduced the notion of `having 
$\Sigma$-unipotent monodromy' for an overconvergent isocrystal on 
$(X,\ol{X})/K$ and proved that an overconvergent isocrystal on 
$(X,\ol{X})/K$ has $\Sigma$-unipotent monodromy if and only if 
it can be extended to an isocrystal on log convergent site 
$((\ol{X},M_{\ol{X}})/O_K)_{\conv}$ with exponents in $\Sigma$. 
In this paper, we prove a `cut-by-curves criterion' for 
an overconvergent isocrytal on $(X,\ol{X})/K$ to have 
$\Sigma$-unipotent monodromy. \par 
Let us give the precise statement of our main theorem. 
For an open immersion of smooth $k$-curves $C \hra \ol{C}$ 
such that $P:=\ol{C}-C = \coprod_{i=1}^s P_i$ is a simple normal 
crossing divisor (= disjoint union of closed points $P_i$ 
whose residue fields are separable over $k$), 
denote the log structure on $\ol{C}$ 
associated to $P$ by $M_{\ol{C}}$. When we are given an exact 
locally closed immersion $\iota: (\ol{C},M_{\ol{C}}) \hra 
(\ol{X},M_{\ol{X}})$ with $(\ol{C},M_{\ol{C}})$ as above, 
$\iota$ induces a well-defined morphism 
of sets $\{1,...,s\} \lra \{1,...,r\}$, which we denote also by 
$\iota$, by the rule $\iota(P_{i}) \subseteq Z_{\iota(i)}$. 
Then we define $\iota^*\Sigma$ by $\iota^*\Sigma := 
\prod_{i=1}^s\Sigma_{\iota(i)} \subseteq \Z_p^s$. 
Also, $\iota$ induces a morphism of pairs 
$(C,\ol{C}) \lra (X,\ol{X})$ which is also denoted by $\iota$ again. 
Therefore, for an overconvergent isocrystal $\cE$ on $(X,\ol{X})/K$, 
we can define the pull-back $\iota^*\cE$ of $\cE$ by $\iota$, 
which is an overconvergent isocrystal on $(C,\ol{C})/K$. 
With this notation, we can state our main theorem as follows: 

\begin{thm}\label{mainthm}
Let $K,k,X,\ol{X},M_{\ol{X}}, \Sigma$ be as above and assume that $k$ is 
uncountable. Then, for an overconvergent isocrystal $\cE$ on 
$(X,\ol{X})/K$, the following two conditions are equivalent$:$ 
\begin{enumerate}
\item 
$\cE$ has $\Sigma$-unipotent monodromy. 
\item 
For any $(C,\ol{C}), M_{\ol{C}}$ as above and for any 
exact locally closed immersion $\iota: (\ol{C},\allowbreak 
M_{\ol{C}}) \hra 
(\ol{X},M_{\ol{X}})$, $\iota^*\cE$ has 
$\iota^*\Sigma$-unipotent monodromy. 
\end{enumerate}
\end{thm} 

Note that Theorem \ref{mainthm} is a $p$-adic analogue of a version of 
cut-by-curves criterion for 
regular singularity of an integrable connection on a smooth variety 
over a field of characteristic $0$, which is shown in 
\cite[II 4.4]{deligne}, \cite[I 3.4.7]{ab}. \par 
Here we briefly sketch the proof. Since the implication 
(1)$\,\Longrightarrow\,$(2) is a rather easy consequence of 
the results in \cite{sigma}, the essential part is to prove 
(1) assuming (2). The condition (1) is equivalent to the condition 
that, for any $i$, the $\nabla$-module associated to $\cE$ 
on $p$-adic annulus around the generic point of $Z_i$ has highest 
ramification break $0$ and that any entry of its exponent 
(in the sense of Christol-Mebkhout) is contained in $\Sigma_i$. 
On the other hand, the condition (2) implies that, 
for any $i$ and for any separable closed point $z \in Z$, 
the $\nabla$-module associated to $\cE$ 
on $p$-adic annulus around $z$ transverse to $Z_i$ has 
highest ramification $0$ and that any component of its exponent 
is contained in $\Sigma_i$. We will prove the latter 
condition implies the former. In fact, we will prove 
stronger assertions: We prove a kind of `cut-by-curves criteria' 
on solvability, highest ramification break and 
exponent of $\nabla$-modules. \par 
The content of each section is as follows: In Section 1, 
we give a review of the notions related to 
(log-)$\nabla$-modules on rigid analytic spaces 
due to Christol, Mebkhout, Dwork and Kedlaya which we need 
for the proof and recall the results in our previous paper 
\cite{sigma}. In Section 2, we prove cut-by-curves criteria 
for solvability, highest ramification break and exponent of 
$\nabla$-modules and using these, 
we give a proof of the main theorem. We also prove a variant of the 
main theorem (Corollary \ref{maincor}) 
which treats the case where $k$ is not necessarily 
uncountable. \par 
The author would like to thank to Fabien Trihan, who suggested to 
include Corollary \ref{maincor}. 
The author is partly supported by Grant-in-Aid for Young Scientists (B), 
the Ministry of Education, Culture, Sports, Science and Technology of 
Japan and JSPS Core-to-Core program 18005 whose representative is 
Makoto Matsumoto. \par

\section*{Convention}

Throughout this paper, $K$ is a fixed complete discrete 
valuation field with residue field $k$ of 
characteristic $p>0$ and let $|\cdot|: K \lra \R_{\geq 0}$ be the 
fixed absolute value on $K$. Let 
$\Gamma^*$ be $(\Q \otimes_{\Z} |K^{\times}|) \cup \{0\}$ and let 
$O_K$ be the the ring of integers of $K$. 
For a $p$-adic formal scheme $\cP$ topologically of finite type over $O_K$, 
we denote the associated rigid space by $\cP_K$. 
A $k$-variety means a reduced separated scheme of finite type 
over $k$. A closed point in a $k$-variety $X$ is called a separable 
closed point if its residue field is separable over $k$. \par 
We use freely the notion 
concerning isocrystals on log convergent site and overconvergent 
isocrystals. For the former, see \cite{pi1II}, \cite{relI} and 
\cite[\S 6]{kedlayaI}. For the latter, see \cite{berthelotrig} and 
\cite[\S 2]{kedlayaI}. 

\section{Preliminaries}

In this section, 
we give a review of several notions and facts related to 
(log-)$\nabla$-modules on rigid analytic spaces 
due to Christol, Mebkhout, Dwork and Kedlaya which we need 
for the proof of our main result. Also, we recall 
the results in our previous paper \cite{sigma}. \par 

\subsection{Differential modules}

First we give a definition of differential modules and cyclic vectors, 
following 
\cite[1.1.4]{kedlayaswanI}: 

\begin{defn}\label{def:dm}
Let $L$ be a field of characteristic $0$ endowed with 
non-Archimedean norm $|\cdot|$ and a non-zero differential $\pa$. 
\begin{enumerate}
\item 
A differential module on $L$ is a finite-dimensional 
$L$-vector space $V$ endowed with an action of $\pa$ satisfying 
the Leibniz rule. 
\item 
Let $V$ be a differential module on $L$. $\v\in V$ is called 
a cyclic vector if and only if $\v, \pa(\v), ..., \pa^{\dim V -1}(\v)$ 
forms a basis of $V$ as $L$-vector space. 
\end{enumerate}
\end{defn} 

It is known that there exists a cyclic vector for any non-zero 
differential module $V$ on $L$. \par 
When $V$ is a differential module on $L$, we can define the operator 
norm $|\pa^n|_V$ of $\pa^n$ on $V$ ($n \in \N$) by fixing a basis of $V$. 
Then we define 
the spectral norm $|\pa|_{V,\sp}$ of $\pa$ on $V$ by 
$|\pa|_{V,\sp} := \lim_{n\to\infty} |\pa^n|^{1/n}$. 
(In particular, we can define $|\pa|_L, |\pa|_{L,\sp}$.) \par 
By using cyclic vector, we can calculate $|\pa|_{V,\sp}$ 
in certain case by the following proposition: 

\begin{prop}[{\cite[Theorem 1.5]{cd}}]
Let $V$ be a differential module on $L$, $\v$ a cyclic vector 
of $V$ and assume that we have the equation 
$\pa^{\dim V}(\v) = \sum_{i=0}^{\dim V-1}a_i\pa^i(\v) \allowbreak 
\,(a_i \in L)$. 
Let $r$ be the least slope of the lower convex hull of the set 
$\{(-i,-\log |a_i|) \,\vert\, \allowbreak 0 \leq i \leq \dim V-1 \}$ in 
$\R^2$. Then we have 
$\max (|\pa|_L, |\pa|_{V,\sp} ) = \max (|\pa|_L, e^{-r})$. 
\end{prop} 

\subsection{$\nabla$-modules and generic radii of convergence}

Let $K$ be as in Convention and let 
$L$ be a field of characteristic $0$ containing $K$ 
complete with respect to a norm 
(denoted also by $|\cdot|$) which extends the given absolute 
value of $K$. We define the notion of $\nabla$-modules on 
rigid spaces over $L$ as follows, following {\cite[2.3.4]{kedlayaI}}:

\begin{defn} 
Let $f:\fX \lra \fY$ be a morphism of rigid spaces over $L$. A 
$\nabla$-module on $\fX$ relative to $\fY$ is a coherent $\cO_{\fX}$-module 
$E$ endowed with an integrable $f^{-1}\cO_{\fY}$-linear connection 
$\nabla: E \lra E \otimes_{\cO_{\fX}} \Omega^1_{\fX / \fY}$. In the case 
$\fY=\Spm L$, we omit the term `relative to $\fY$'. 
\end{defn} 

We call a subinterval $I \subseteq [0,\infty)$ 
aligned if any endpoint of $I$ at which it is closed is contained in 
$\Gamma^*$. For an aligned interval $I$, we 
define the rigid space $A^n_L(I)$ by  
$A^n_L(I) := \{(t_1, ..., t_n) \in {\Bbb A}^{n,\an}_L 
\,\vert\, \forall i, |t_i| \in I \}.$ \par 
For $\rho \in [0,+\infty)$, we denote the completion of 
$L(t)$ with respect to $\rho$-Gauss norm $|\cdot|_{\rho}$ 
by $L(t)_{\rho}$. This is 
a differential field with continuous differential operator $\pa$ 
satisfying $\pa(t)=1$. It is easy to show the equalities 
$|\pa|_{\rho} = \rho^{-1}$, $|\pa|_{\rho,\sp}=p^{-1/(p-1)}\rho^{-1}$. 
For any closed aligned interval 
$I$ containing $\rho$, we have the 
natural inclusion $\Gamma(A^1_L(I),\cO) \hra L(t)_{\rho}$. 
For an aligned interval $I$, a $\nabla$-module $E$ on 
$A^1_L(I)$ and $\rho \in I$, we define the differential module 
$E_{\rho}$ 
on $L(t)_{\rho}$ by 
$E_{\rho}:= \Gamma(A^1_L(I'),E) \otimes_{\Gamma(A^1_L(I'),\cO)} 
L(t)_{\rho}$ endowed with the action 
\begin{equation}\label{partial}
 \pa: E_{\rho} \os{\nabla_{\rho}}{\lra} E_{\rho}dt \lra E_{\rho}, 
\end{equation}
where $I'$ is any closed aligned subinterval of $I$ containing 
$\rho$, $\nabla_{\rho}$ is the connection on $E_{\rho}$ induced by 
$\nabla$ and the last map in \eqref{partial} is the map 
$xdt \mapsto x$. Then we define the notion of generic 
radius of convergence as follows (\cite[4.1]{cmsurvey}): 

\begin{defn}
Let $I \subseteq [0,+\infty)$ be an aligned interval and 
let $E$ be a $\nabla$-module on $A^1_L(I)$. Then we define the 
generic radius of convergence $R(E,\rho)$ of $E$ at the radius $\rho$ by 
$$ R(E,\rho) := 
\rho |\pa|_{L(t)_{\rho},\sp}/|\pa|_{E_{\rho},\sp} = 
p^{-1/(p-1)}/|\pa|_{E_{\rho},\sp}. $$
If we fix a basis $\e := (\e_1,...,\e_{\mu})$ of $E_{\rho}$ and define 
$G_n$ to be the matrix expression of $\pa^n$ with respect to the basis 
$\e$, then we have the following equivalent definition$:$ 
$$ R(E,\rho) = \min(\rho, \varliminf_{n\to\infty}|G_n/n!|_{\rho}^{-1/n}). $$
$E$ is said to satisfy the Robba condition if $R(E,\rho)=\rho$ for any 
$\rho \in I$. 
\end{defn}

As for the behavior of $R(E,\rho)$, the following proposition is known. 

\begin{prop}[{\cite[8.6]{cmsurvey}}]\label{f}
Let $I \subseteq [0,+\infty)$ be an aligned interval and 
let $E$ be a $\nabla$-module on $A^1_L(I)$ of rank $\mu$. 
Let $f: -\log I = \{-\log x \,\vert\, x \in I\} \lra \R$ be the 
function defined by $f(r):=-\log R(E,e^{-r})$. Then $f$ is continuous 
piecewise affine linear convex function whose slopes are in 
$\Z\!\cdot\!(\mu!)^{-1}$. 
\end{prop}

For $\lam \in [0,1)\cap\Gamma^*$, we call a $\nabla$-module $E$ on 
$A_L^1[\lam,1)$ solvable if $\displaystyle\lim_{\rho\to1^-}R(E,\rho)=1$. 
As a corollary to Proposition \ref{f}, we have the following: 

\begin{cor}\label{elementary}
Let $\lam \in [0,1)\cap\Gamma^*$ and let $E$ be a $\nabla$-module on 
$A_L^1[\lam,1)$. Then$:$ 
\begin{enumerate}
\item
The following conditions are equivalent.
\begin{enumerate}
\item 
$E$ is solvable. 
\item 
There exist $b \geq 0$ and a strictly 
increasing sequence $\{\rho_m\}_{m \in \N} 
\subseteq 
[\lam, 1)$ with 
$\dlim_{m\to\infty}\allowbreak \rho_m =1$ such that $R(E,\rho_m)=\rho_m^{b+1}$ 
for all $m \in \N$. 
\item 
There exists $b \geq 0$ 
such that $R(E,\rho)=\rho^{b+1}$ for any $\rho$ sufficiently close to $1$. 
\end{enumerate}
If these conditions are satisfied, we call $b$ the highest ramification 
break of $E$. 
\item 
If $E$ is solvable, the following are equivalent. 
\begin{enumerate}
\item 
$E$ has highest ramification break $b$. 
\item
There exist $\rho_1, \rho_2 \in [\lam,1)$ with $\rho_1 < \rho_2$ 
such that $R(E,\rho_i)=\rho_i^{b+1}\,(i=1,2)$. 
\end{enumerate}
\item 
The following conditions are equivalent.
\begin{enumerate}
\item 
$E$ is not solvable. 
\item 
There exists a strictly 
increasing sequence $\{\rho_m\}_{m \in \N} 
\subseteq [\lam, 1)$ with 
$\dlim_{m\to\infty}\allowbreak \rho_m =1$ such that 
$\displaystyle\varlimsup_{m\to\infty} R(E,\rho_m) < 1$. 
\item 
For any strictly 
increasing sequence $\{\rho_m\}_{m \in \N} 
\subseteq [\lam, 1)$ with 
$\dlim_{m\to\infty}\allowbreak \rho_m =1$, 
$\displaystyle\varlimsup_{m\to\infty} R(E,\rho_m) < 1$. 
\end{enumerate}
\end{enumerate}
\end{cor}

\begin{proof}
Although the proof is well-known, we give a proof for the reader's 
convenience. Let us put $f(r):=-\log R(E,e^{-r})$. \par 
First we prove (1). (c)$\,\Longrightarrow\,$(a) and 
(c)$\,\Longrightarrow\,$(b) are obvious. 
Let us prove (b)$\,\Longrightarrow\,$(c). Put $r_m := -\log \rho_m$. 
Then we have $f(r)=(b+1)r$ for $r \in (0,r_1]$: Indeed, 
if we have $f(r)>(b+1)r$ for some $r \in (0,r_1]$ with $r_m < r$, 
this contradicts the convexity of $f$ on $[r_m,r_1]$. On the other hand, 
if we have $f(r)<(b+1)r$ for some $r \in (0,r_1]$, 
this contradicts the convexity of $f$ on $[r,r_0]$. So we have proved (c). 
\par 
We prove (a)$\,\Longrightarrow\,$(c). If the slope of $f$ is 
less than $1$ at some $r_0 \in (0,-\log \lam]$, then $f(r)=pr+q$ for 
some $p<1, q \in \R$ around $r_0$. By definition of $f$, we have 
$r_0 \leq f(r_0) = pr_0+q$. So $q \geq (1-p)r_0>0$. Then, by convexity of 
$f$, we obtain the inequality $\displaystyle\varlimsup_{r\to0^+}f(r) \geq 
\displaystyle\varlimsup_{r\to0^+}(pr+q) = q >0$ 
and this contradicts the solvability of 
$E$. So the slope of $f$ is equal to or greater than $1$ on 
$(0,- \log \lam]$. Since the slope of $f$ is always in 
$\Z\!\cdot\!(\mu!)^{-1}$ and 
decreasing as $r\to 0^+$, it is stationary around $0$, that is, 
we have $f(r) = pr+q$ for some $p\geq 1,q \in \R$ on some $(0,r_1]$. 
Then, by the solvability of $E$, we have $p\geq 1$ and $q=0$, which imply (c). 
So we have proved (1).\par 
Next we prove (2). (a)$\,\Longrightarrow\,$(b) is obvious. 
Let us prove (b)$\,\Longrightarrow\,$(a). Let us assume (b) and assume that 
$E$ has highest ramification break $b'$. 
Put $r_i := - \log \rho_i$ so that $f(r_i)=(b+1)r_i\,(i=1,2)$. Then, 
if we have $b'<b$, there exists $r_3<r_2$ with $f(r_3)<(b+1)r_3$ and 
this contradicts the convexity of $f$ on $[r_3,r_1]$. On the other hand, 
if we have $b'>b$, there exists $r_4<r_3<r_2$ with $f(r_i)=(b'+1)r_i\,(i=3,4)$ 
and this contradicts the convexity of $f$ on $[r_4,r_2]$. Hence we have 
$b'=b$, that is, the assertion (a). So we have proved (2). \par 
Finally we prove (3). Since 
 (c)$\,\Longrightarrow\,$(b) and (b)$\,\Longrightarrow\,$(a) are obvious, 
it suffices to prove (a)$\,\Longrightarrow\,$(c). To do this, first we 
prove the claim that $f(r)$ has the form $f(r)=pr+q$ for some $p\in\R, q>0$ on 
some $[s,t] \subseteq [0,-\log \lambda)\,(s<t)$. 
If the slope of $f$ is less than 
$1$ at some $r_0 \in (0,-\log \lam]$, we have 
$f(r)=pr+q$ for some $p<1, q>0$ around $r_0$, as we saw in the proof of (1). 
Hence the claim is true in this case. If the slope of 
$f$ is equal to or greater than $1$ on $(0,- \log \lam]$, 
we have $f(r) = pr+q$ for some $p\geq 1, q \in \R$ on some $(0,r_1]$, which 
we also saw in the proof of (1). Then, if we have $q<0$, 
we have $(1-p)r > q$ for some $r$ and this implies the inequality 
$r>f(r)$, which contradicts the definition of $f$. Also, if $q=0$, 
then $E$ would be solvable and this is also contradiction. Hence we 
have $q>0$ and the claim is true also in this case. \par 
By the above claim and the convexity of $f$, we have 
$f(r) \geq pr+q$ for $r \in (0,s]$. Hence, if we put $r_m := -\log \rho_m$, 
we have $\displaystyle\varliminf_{m\to\infty}f(r_m) 
\geq \displaystyle\varliminf_{m\to\infty}(pr_m+q) 
= q >0$. This implies the assertion (a) and so we are done. 
\end{proof}

\subsection{Frobenius antecedent}

Let 
$L$ be a field of characteristic $0$ containing $K$ and 
a primitive $p$-th root of unity $\zeta$ 
complete with respect to a norm 
(denoted also by $|\cdot|$) which extends the given absolute 
value of $K$. For an aligned interval $I \subseteq [0,1)$, 
let us put $I^p := \{a^p \,\vert\, a \in I\}$ and 
let $\varphi_L: A_L^1(I) \lra A_L^1(I^p)$ be the morphism over $L$ 
induced by $t \mapsto t^p$, 
where $t$ denotes the coordinate of the annuli. 

\begin{propdefn}[{\cite[10.4.1, 10.4.4]{kedlayabook}}]\label{frobant}
Let $I \subseteq (0,+\infty)$ be a closed aligned interval and let $E_L$ be a 
$\nabla$-module on $A^1_L(I)$ satisfying $R(E_L,\rho) > p^{-1/(p-1)}\rho$ 
for any $\rho \in I$. Then there exists a unique $\nabla$-module $F_L$ 
on $A_L^1(I^p)$ such that $R(F_L,\rho) > p^{-p/(p-1)}\rho$ for any 
$\rho \in I^p$ and that $\varphi_L^*F_L = E_L$. Moreover, we have 
$R(F_L,\rho^p)= R(E_L,\rho)^p$ for any $\rho \in I$. 
We call this $F_L$ the Frobenius antecedent of $E_L$. 
\end{propdefn}

\begin{cor}
Let $I \subseteq (0,+\infty)$ be a closed aligned interval and 
let $F_L$ be a $\nabla$-module on $A^1_L(I^p)$ with 
$R(F_L,\rho^p) > p^{-p/(p-1)}\rho^p$ for $\rho \in I$. Then 
$F_L$ is the Frobenius antecedent of $\varphi_L^*F_L$. 
\end{cor}

\begin{rem} 
In \cite{kedlayabook}, the above proposition is proved also in the case 
where $L$ does not necessarily contain a primitive $p$-th root of 
unity, and also in the case where $I$ contains $0$. 
\end{rem}

We explain how to construct $F_L$, following \cite{kedlayabook}. 
Let $\pa$ be the composite 
$E_L \os{\nabla}{\lra} E_L dt \os{=}{\lra} E_L$ (where the second map 
sends $xdt$ to $x$ for $x \in E_L$). 
Let us define the action of 
the group $\mu_p = \{\ul{\zeta}^m\,\vert\,0 \leq m \leq p-1\}$ 
of the $p$-th roots of unity in $L$ on $E_L$ by 
$\ul{\zeta}^m(x) := 
\dsum_{i=0}^{\infty}\dfrac{((\zeta^m-1)t)^i}{i!} \pa^i(x)$
and let $P_j: E_L \lra E_L \,(0 \leq j \leq p-1)$ be the map defined 
by $P_j(x) := p^{-1} \dsum_{i=0}^{p-1} \zeta^{-ij}(\ul{\zeta}^i(x))$. 
Then, it is straightforward to see that $P_j$'s satisfy 
$P_j^2 = P_j, P_jP_{j'}=0 \,(j\not=j), \sum_{j=0}^{p-1}P_j\allowbreak=\id$. 
Hence we have the isomorphism $E \os{\cong}{\lra} 
\bigoplus_{j=0}^{p-1}\im P_j$. 
Also, it is easy to see that each $\im P_j$ has a natural structure 
of $\cO_{A_L^1(I^p)}$-module and that 
the map $x \mapsto t^jx$ induces the isomorphism 
$\im P_0 \os{\cong}{\lra} \im P_j$ of  $\cO_{A_L^1(I^p)}$-modules. 
Moreover,  $\cO_{A_L^1(I)}$ is a free $\cO_{A_L^1(I^p)}$-module with 
basis $1,t,...,t^{p-1}$. These facts imply that we have the canonical 
isomorphism 
$\oplus_{j=0}^{p-1}\im P_j \os{=}{\lla} \cO_{A_L^1(I)} 
\otimes_{\cO_{A_L^1(I^p)}} \im P_0$. So, if we put $F_L := \im P_0$, we 
have $\varphi^*_LF_L = E_L$. This $F_L$, endowed with the $\nabla$-module 
structure $x \mapsto \dfrac{1}{pt^{p-1}}\pa(x)d(t^p)$ (where we denote 
the coodinate of $A^1_K(I^p)$ by $t^p$), gives the desired 
$\nabla$-module on $A_L^1(I^p)$. (For the detail, see 
\cite[10.4.1]{kedlayabook}). \par 
We will slightly generalize the above construction. Let 
$\fX := \Spm A$ be a smooth affinoid rigid space over $K$ 
which admits an injection $A \hra L$ such that the norm on $L$ restricts to 
the supremum norm on $\fX$. 
For an aligned interval $I \subseteq [0,1)$, we can define the 
morphism $\varphi: \fX \times A_K^1(I) \lra \fX \times A_K^1(I^p)$ 
over $\fX$ by $t \mapsto t^p$. \par 
Let $I, E_L$ be as in 
Proposition-Definition \ref{frobant} and assume that $E_L$ is obtained from 
a $\nabla$-module $E$ on $\fX \times A^1_K(I)$ relative to $\fX$ 
via the `pull-back' by $A^1_L(I) \lra \fX \times A^1_K(I)$. Then we have the 
following: 

\begin{prop}\label{frobantrel}
With the situation above, assume moreover that $K$ contains $\zeta$. 
Then there exists a $\nabla$-module $F$ on 
$\fX \times A^1_K(I^p)$ relative to $\fX$ 
with $\varphi^*F \cong E$ which induces $F_L$ 
in Proposition-Definition \ref{frobant} via 
the `pull-back' by $A^1_L(I) \lra \fX \times A^1_K(I)$.
\end{prop}

\begin{proof}
We can define the maps $P_j:E \lra E \, (0 \leq j \leq p-1)$ in the same 
way as in the case of $E_L$ and we can prove the isomorphisms 
$$ E \os{\cong}{\lra} \bigoplus_{j=0}^{p-1} \im P_j \os{\cong}{\lla} 
\cO_{\fX \times A_K^1(I)} 
\otimes_{\cO_{\fX \times A_K^1(I^p)}} \im P_0 $$
in the same way. Then it suffices to put $F := \im P_0$. 
\end{proof} 

\subsection{Exponent in the sense of Christol-Mebkhout}

In this subsection, we give a review of the exponent 
(in the sense of Christol-Mebkhout) for $\nabla$-modules on annuli 
satisfying the Robba condition, following \cite[10--12]{cmsurvey}. \par 
For $\alpha \in \Z_p$ and $h \in \N$, let $\alpha^{(h)}$ be the 
representative of $\alpha$ modulo $p^h$ in $[(1-p^h)/2, (1+p^h)/2)$. 
For $\mu \in \N$ and $\Delta = (\Delta_i) \in \Z_p^{\mu}$, 
let us put $\Delta^{(h)} := (\Delta_i^{(h)}) \in \Z^{\mu}$, 
$\sigma(\Delta) := (\Delta_{\sigma(i)}) \in \Z_p^{\mu} \, (\sigma \in 
{\frak{S}}_{\mu}).$ Let $|\cdot|_{\infty}$ be the usual absolute value 
on $\Z$ and for $\Delta = (\Delta_i) \in \Z^{\mu}$, 
put $|\Delta|_{\infty}:= \max |\Delta_i|_{\infty}$. \par 
With the above notation, we define the equivalence relation 
$\os{e}{\sim}$ on $\Z_p^{\mu}$ as follows: For $\Delta, \Delta' \in 
\Z_p^{\mu}$, $\Delta \os{e}{\sim} \Delta'$ if and only if 
there exists a sequence $(\sigma_h)_{h\in \N}$ of elements in 
${\frak{S}}_{\mu}$ such that 
$|{\Delta'}^{(h)}-\Delta^{(h)}|_{\infty}/h \, (h\in \N)$ is bounded. \par 
Let $K$ be as in Convention and 
let $L$ be a field of characteristic $0$ containing $K$ 
complete with respect to a norm 
(denoted also by $|\cdot|$) which extends the given absolute 
value of $K$. Let $L_{\infty}$ be the completion of $L(\mu_{p^{\infty}})$, 
where $\mu_{p^{\infty}}$ is the set of $p$-power roots of unity. 
Let $I \subseteq (0,+\infty)$ be a closed aligned interval and 
let $E$ be a $\nabla$-module on $A_L^1(I)$ of rank $\mu$ satisfying the 
Robba condition. 
($E$ is automatically a free module on $A_L^1(I)$ of rank $\mu$.) 
Let $\pa$ be the composite 
$E \os{\nabla}{\lra} E dt \os{=}{\lra} E$ (where the second map 
sends $xdt$ to $x$ for $x \in E$). Fixing a basis 
$\e := (\e_1, ..., \e_{\mu})$ of $E$, we define the 
resolvent $Y_{\e}(x,y)$ associated to $\e$ by 
$Y_{\e}(x,y) := \sum_{n=0}^{\infty} G_n(y)\dfrac{(x-y)^n}{n!}$, where 
$G_n \in \Mat_{\mu}(\cO_{A_L^1(I)})$ denotes the matrix expression 
of $\pa^n$ with respect to the basis $\e$. 
($Y_{\e}(x,y)$ is defined on $\{(x,y) \in \Af^{2,\an}_L \,\vert\, 
|y| \in I, |x-y| < |y|\}$.) With these notations, 
we define the exponent of $E$ as follows: 

\begin{thmdefn}[{\cite[11.3]{cmsurvey}}]\label{exponent}
Let $I, E, \mu$ be as above, take a basis $\e$ of $E$ and let 
$Y_{\e}(x,y)$ be the resolvent of $E$ associated to $\e$. 
Then the set of elements $\Delta$ in $\Z_p^{\mu}$ for which 
there exist a sequence $(S_h)_{h \in \N}$ in 
$\Mat_{\mu}(\cO_{A_{L_{\infty}}^1(I)})$ 
and constants $c_1,c_2 >0$ satisfying the 
conditions 
\begin{enumerate}
\item 
$\zeta^{\Delta} S_h(x) = S_h(\zeta x) Y_{\e}(\zeta x,x)$ for 
any $h \in \N$ and any $\zeta$ with $\zeta^{p^h}=1$, where 
$\zeta^{\Delta}$ denotes the diagonal $\mu \times \mu$ matrix with 
diagonal entries $\zeta^{\Delta_i} \,(1 \leq i \leq \mu)$. 
\item 
$|S_h|_{\rho} \leq c_1^h$ for any $h \in \N$ and $\rho \in I$, where 
$|\cdot|_{\rho}$ denotes the $\rho$-Gauss norm. 
\item 
There exists some $\rho_0 \in I$ such that 
$|\det (S_h)|_{\rho_0} \geq c_2$ for any $h \in \N$. 
\end{enumerate}
is non-empty, independent of the choice of $\e$ and contained in 
one equivalence class for the relation $\os{e}{\sim}$. 
We call this class the exponent of $E$ and denote it by $\Exp(E)$. 
\end{thmdefn}

We recall the construction of $\Delta$ and $(S_h)_{h \in \N}$, 
following \cite[11.2.1]{cmsurvey}. (See also \cite{dwork}.) 
For $h \in \N$ and $\Delta_h \in (\Z/p^h\Z)^{\mu}$, 
let us define $S_{h,\Delta_h}(x)$ by 
$S_{h,\Delta_h} := 
p^{-h} \dsum_{\zeta^{p^h}=1} \zeta^{-\Delta_h}Y_{\e}(\zeta x,x)$. 
Then we have 
$$ \det (S_{h,\Delta_h}) = \dsum_{\Delta' \in (\Z/p^{h+1}\Z)^{\mu}, 
\equiv \Delta_{h} \,{\rm mod}\, p^h} \det (S_{h+1,\Delta'}). $$
Fix any $\rho_0 \in I$. Then, from the above equation, we can choose 
$\Delta \in \Z_p^{\mu}$ such that, if we put $\Delta_h := \Delta \,{\rm mod}\, 
p^h$, we have the inequalities $|\det (S_{h,\Delta_{h}})|_{\rho_0} 
\leq |\det (S_{h+1,\Delta_{h+1}})|_{\rho_0}$ 
for any $h \in \N$. Then, if we put 
$S_h := S_{h,\Delta_h}$, 
this $\Delta$ and $(S_h)_{h \in \N}$ satisfy the conditions (1), (2), (3) 
in Theorem-Definition \ref{exponent}. (For detail, see 
\cite{cmsurvey} and the references which are quoted there.) \par 
Let $\lam \in (0,1)\cap\Gamma^*$ and let $E$ be a $\nabla$-module on 
$A^1_L[\lam,1)$ with highest ramification break $0$. Then there exists some 
$\lam' \in [\lam,1)\cap\Gamma^*$ such that, for any 
closed aligned subinterval $I$ in $[\lam',1)$, 
$E$ satisfies the Robba condition on $A_L^1(I)$ (see Corollary 
\ref{elementary}). Hence we can define the exponent $\Exp(E)$ of $E$ and 
it is known that this defininition of $\Exp(E)$ 
is independent of the choice of $\lam'$ and $I$. \par 
As for the relation of the equivalence $\os{e}{\sim}$ with 
$p$-adic non-Liouvilleness, we have the following: 

\begin{prop}[{\cite[10.5]{cmsurvey}}]\label{nld}
Let $\Delta = (\Delta_i), \Delta'=(\Delta'_i) \in \Z_p^{\mu}$ 
and assume that $\Delta_i-\Delta_j$ are $p$-adically non-Liouville for 
any $i,j$. Then we have $\Delta \os{e}{\sim} \Delta'$ if and only if 
there exists an element $\sigma \in {\frak{S}}_{\mu}$ with 
$\Delta' - \sigma(\Delta) \in \Z^{\mu}$. 
\end{prop}

\subsection{Overconvergent isocrystals and log-extendability}

In this subsection, we recall the notion of 
overconvergent isocrystals having $\Sigma$-unipotent monodromy and 
isocrystals on log convergent site with exponents in $\Sigma$ for 
$\Sigma \subseteq \Z_p^r$. Also, we recall several results 
proved in \cite{sigma}. 
In this subsection, we fix an open immersion $X \hra \ol{X}$ of 
smooth $k$-varieties with $Z := \ol{X} - X = \bigcup_{i=1}^rZ_i$ a simple 
normal crossing divisor and let us denote the log structure on $\ol{X}$ 
associated to $Z$ by $M_{\ol{X}}$. First we recall the notion concerning 
frames. 

\begin{defn}[{\cite[2.2.4, 4.2.1]{kedlayaI}}] 
\begin{enumerate}
\item 
A frame $($or affine frame$)$ is a tuple $(U,\ol{U},\allowbreak \cP,i,j)$, 
where $U,\ol{U}$ are $k$-varieties, $\cP$ is a $p$-adic affine formal scheme 
topologically of finite type over $O_K$, $i:\ol{U} \hra \cP$ is a 
closed immersion over $O_K$, $j: U \hra \ol{U}$ is an open immersion 
over $k$ such that $\cP$ is formally smooth over $O_K$ on a neighborhood 
of $X$. We say that the frame encloses the pair $(U,\ol{U})$. 
\item 
A small frame is a frame $(U,\ol{U},\cP,i,j)$ such that 
$\ol{U}$ is isomorphic to $\cP \times_{\Spf O_K} \Spec k$ 
via $i$ and that there exists 
an element $f \in \Gamma(\ol{U},\cO_{\ol{U}})$ with $U=\{f\not=0\}$. 
\end{enumerate}
\end{defn} 

\begin{defn}[{\cite[3.3]{sigma}}] 
Let $U \hra \ol{U}$ be an open immersion of smooth $k$-varieties such 
that $Y := \ol{U}-U$ is a simple normal crossing 
divisor and let $Y = \bigcup_{i=1}^r Y_i$ be a decomposition of 
$Y$ such that $Y=\bigcup_{\scriptstyle 1 \leq i \leq r \atop 
\scriptstyle Y_i \not= \emptyset} Y_i$ gives the decomposition of 
$Y$ into irreducible components. 
A standard small frame enclosing $(U,\ol{U})$ 
is a small frame $\bP := (U,\ol{U},\cP,i,j)$ enclosing $(U,\ol{U})$ 
which satisfies the following condition$:$ 
There exist $t_1, ..., t_r \in \Gamma(\cP,\cO_{\cP})$ such that, 
if we denote the zero locus of $t_i$ in $\cP$ by $\cQ_i$, 
each $\cQ_i$ is irreducible $($possibly empty$)$ and that 
$\cQ = \bigcup_{i=1}^r\cQ_i$ is a relative simple normal crossing 
divisor of $\cP$ 
satisfying $Y_i = \cQ_i \times_{\cP} \ol{U}$. 
We call a pair $(\bP, (t_1,...,t_r))$ a charted standard small frame. 
When $r=1$, we call $\bP$ a smooth standard small frame and 
the pair $(\bP,t_1)$ a charted smooth standard small frame. 
\end{defn} 

Let $\ol{U}$ be an open subscheme of $\ol{X}$, let $U := X \cap \ol{U}$ and 
let $((U,\ol{U},\cP,i,j),t)$ be a charted smooth standard small frame. Then 
an overconvergent isocrystal $\cE$ on $(X,\ol{X})/K$ induces a 
$\nabla$-module on $\fU_{\lam} := 
\{x \in \cP_K \,\vert\, |t(x)| \geq \lam\}$ for some 
$\lam \in (0,1)\cap \Gamma^*$, which we denote by $E_{\cE} = 
(E_{\cE},\nabla)$. 
Since $\fU_{\lam}$ contains a relative 
annulus $\cQ_K \times A^1_K[\lam,1)$, we can restrict $E_{\cE}$ to this space. 
Furthurmore, if $\cQ_K$ is non-empty and if 
we are given an injection $\Gamma(\cQ_K,\cO_{\cQ_K}) \hra L$ 
into a field $L$ endowed with a complete norm which restricts to the supremum 
norm on $\cQ_K$ (e.g., when $L$ is the completion of the fraction field of 
$\Gamma(\cQ_K,\cO_{\cQ_K})$ with respect to the supremum norm), 
we can further restrict $E_{\cE}$ to $A_L^1[\lam,1)$, 
which we denote by $E_{\cE,L}$. Then we have the following proposition:

\begin{prop}\label{ocsolv}
With the above notation, the $\nabla$-module $E_{\cE,L}$ on 
$A^1_L[\lam,1)$ is solvable. 
\end{prop}

\begin{proof}
Since we may shrink $\cP$ (since, for dense open immersion 
$\cQ' \hra \cQ$, induced map $\Gamma(\cQ_K,\cO_{\cQ_K}) \hra 
\Gamma(\cQ'_K,\cO_{\cQ'_K})$ respects the supremum norm), 
we may assume that $\Omega^1_{\cP_K}$ 
is freely generated by $dx_1,...,dx_n,dt$ for some 
$x_1,...,x_n \in \Gamma(\cP_K,\cO_{\cP_K})$. 
Let $\pa:E_{\cE} \lra E_{\cE}$ be the composite 
$$ E_{\cE} \os{\nabla}{\lra} E_{\cE} \otimes \Omega^1_{\fU_{\lam}} 
\os{\pi}{\lra} E_{\cE} dt \os{=}{\lra} E, $$
where $\pi$ sends $dx_i$ (resp. $dt$) to $0$ (resp. $dt$) and the 
last map sends $x dt\,(x \in E_{\cE})$ to $x$. 
Then, by \cite[2.2.13]{berthelotrig} and 
\cite[2.5.6]{kedlayaI}, we have the following: 
For any $\eta <1$, there exists $\lam_0 <1$ such that for any 
$\lam \in (\lam_0,1)\cap \Gamma^*$ and for any $e \in 
\Gamma(\fU_{\lam}, E_{\cE})$, we have 
$\lim_{i\to\infty} \|\pa^i(e)/i!\|\eta^i =0$, where $\|\cdot\|$ denotes 
any Banach norm on $\Gamma(\fU_{\lam}, E_{\cE})$. 
Hence, for any $\eta <1$, there exists $\rho_0 <1$ such that for any 
$\rho \in (\rho_0,1)\cap \Gamma^*$ and for any 
$e \in E_{\cE,L,\rho}$, we have 
$\lim_{i\to\infty} |\pa^i(e)/i!|_{\rho}\eta^i =0$, where 
$|\cdot|_{\rho}$ denotes any norm on $E_{\cE,L,\rho}$ 
induced by the norm on $L(t)_{\rho}$(= the completion of $L(t)$ with 
repsect to the $\rho$-Gauss norm). This implies the inequality 
$p^{1/(p-1)}|\pa|_{\rho,\sp}\eta \leq 1$. Hence, for any $\eta <1$, 
we have $\eta \leq R(E_{\cE,L},\rho)$ for $\rho$ sufficiently close to $1$, 
that is, $E_{\cE,L}$ is solvable. 
\end{proof} 

Next we recall the notion of $\Sigma$-unipotent $\nabla$-modules. 
For aligned interval $I \subseteq (0,+\infty)$ and $\xi \in \Z_p$, 
we define the $\nabla$-module $M_{\xi} := (M_{\xi},\nabla_{M_{\xi}})$ on 
$A^1_K(I)$ (whose coordinate is $t$) as the $\nabla$-module 
$(\cO_{A^1_K(I)}, d + \xi\dlog t)$. 

\begin{defn}[{\cite[1.3]{sigma}}]\label{unipdef}
Let $\fX$ be a smooth rigid space. Let $I \subseteq (0,\infty)$ 
be an aligned interval and fix 
$\Sigma \subseteq \Z_p$. A $\nabla$-module $E$ on 
$\fX \times A_K^1(I)$ is called 
$\Sigma$-unipotent if it admits a filtration 
$$ 0 = E_0 \subset E_1 \subset \cdots \subset E_m=E $$ 
by sub log-$\nabla$-modules whose successive quotients 
have the form 
$\pi_1^*F \otimes \pi_2^*M_{\xi}$ for some $\nabla$-module 
$F$ on $\fX$ and $\xi \in \Sigma$, where 
$\pi_1: \fX \times A_K^1(I) \lra \fX$, $\pi_2:\fX \times A_K^1(I) 
\lra A_K^1(I)$ 
denote the projections. 
\end{defn}

We call a subset 
$\Sigma := \prod_{i=1}^r\Sigma_i$ in $\Z_p^r$ 
$\NID$ $($resp. $\NLD)$ if for any $1 \leq i \leq r$ and 
any $\alpha, \beta \in \Sigma_i$, $\alpha - \beta$ is not 
a non-zero integer $($resp. $p$-adically non-Liouville$)$. 
Under the assumption of $\NID$ and $\NLD$, we have the following 
generization property for $\Sigma$-unipotent $\nabla$-modules: 

\begin{prop}[{\cite[2.4]{sigma}}]\label{generization}
Let $\fX$ be a smooth affinoid rigid space and let 
$\Gamma(\fX,\cO_{\fX}) \hra L$ 
be an injection into a field complete with respect to 
a norm which restricts to 
the supremum norm on $\fX$. Let $I \subseteq (0,1)$ be an open interval, 
let $\Sigma$ be a subset of $\Z_p$ which is $\NID$, $\NLD$ and  
let $E$ be a $\nabla$-module on $X \times A_K^1(I)$ whose restriction 
to $A_L^1(I)$ is $\Sigma$-unipotent. Then $E$ is also $\Sigma$-unipotent. 
\end{prop}

For a subset $\Lambda$ of $\Z_p$ and $\mu \in \N$, let 
$\ol{\Lambda} \subseteq \Z_p^{\mu}/\os{e}{\sim}$ be the image of 
$\Lambda^{\mu}$ by the projection $\Z_p^{\mu} \lra 
\Z_p^{\mu}/\os{e}{\sim}$. Then, by \cite[12.1]{cmsurvey}, we have 
the following characterization of $\Sigma$-unipotence: 

\begin{prop}\label{cmunip}
Let $L$ be a field containing $K$ 
complete with respect to a norm 
which extends the given absolute value of $K$. 
Let $I \subseteq (0,1)$ be an open interval and let $E$ be a 
$\nabla$-module on $A^1_L(I)$. 
Then, for a subset $\Sigma \subseteq \Z_p$ which 
is $\NID$ and $\NLD$, the following are equivalent$:$ 
\begin{enumerate}
\item 
$E$ is $\Sigma$-unipotent. 
\item 
$E$ satisfies the Robba condition and $\Exp(E) \in \ol{\Sigma}$. 
\end{enumerate}
\end{prop}

\begin{proof}
(2)$\,\Longrightarrow\,$(1) follows from Proposition \ref{nld} and 
\cite[12.1]{cmsurvey}. In rank one case, 
(1)$\,\Longrightarrow\,$(2) follows from 
\cite[9.5.2, 13.5.3]{kedlayabook} and the general case follows from 
the rank one case and \cite[4.5, 11.7]{cmsurvey}. 
\end{proof}

Now we recall the definition of 
overconvergent isocrystals having $\Sigma$-unipotent monodromy. 

\begin{defn}[{\cite[3.9]{sigma}}]\label{defsigmon}
Let $(X,\ol{X}),Z=\bigcup_{i=1}^rZ_i$ be as above and 
let $Z_{\sing}$ be the set of singular points of $Z$. 
Let $\Sigma = \prod_{i=1}^r\Sigma_i$ be a 
subset of $\Z_p^r$. Then we say that an overconvergent isocrystal 
$\cE$ on $(X,\ol{X})/K$ has $\Sigma$-unipotent monodromy 
if there exist an affine open covering 
$\ol{X}-Z_{\sing} = \bigcup_{\alpha\in \Delta} \ol{U}_{\alpha}$ and 
charted smooth standard small frames 
$(
(U_{\alpha},\ol{U}_{\alpha},\cP_{\alpha},i_{\alpha},j_{\alpha}), 
t_{\alpha})$ enclosing 
$(U_{\alpha},\ol{U}_{\alpha})$ $(\alpha \in \Delta$, where we put 
$U_{\alpha} := X \cap \ol{U}_{\alpha})$ such that, for 
any $\alpha \in \Delta$, there exists some $\lam \in (0,1)\cap\Gamma^*$ 
such that the $\nabla$-module $E_{\cE, \alpha}$ associated to 
$\cE$ is defined on $\{x \in \cP_{\alpha,K} \,\vert\, 
|t_{\alpha}(x)| \geq \lam\}$ and that the restriction of $E_{\cE,\alpha}$ to 
$\cQ_{\alpha,K} \times A^1_K[\lam,1)$ is $\Sigma_i$-unipotent $($where 
$\cQ_{\alpha}$ is the zero locus of $t_{\alpha}$ in $\cP_{\alpha}$ and $i$ 
is any index with $\ol{U}_{\alpha}\cap Z \subseteq Z_i$, which is unique 
if $\ol{U}_{\alpha}\cap Z$ is non-empty. When $\ul{U}_{\alpha}\cap Z$ is 
empty, we regard this last condition as vacuous one.$)$ 
\end{defn} 

When $\Sigma$ is $\NID$ and $\NLD$, we have the following 
characterization of 
overconvergent isocrystals having $\Sigma$-unipotent monodromy: 

\begin{prop}[{\cite[3.10]{sigma}}]\label{wdsigmon}
Let $(X,\ol{X}), Z=\bigcup_{i=1}^rZ_i, \Sigma$ be as above and assume 
that $\Sigma$ is $\NID$ and $\NLD$. 
Then 
an overconvergent isocrystal $\cE$ on $(X,\ol{X})/K$ has 
$\Sigma$-unipotent monodromy 
if and only if the following condition 
is satisfied$:$ 
For any 
affine open subscheme $\ol{U} \hra \ol{X}-Z_{\sing}$ and 
any charted smooth standard small frame 
$((U,\ol{U},\cP,i,j), t)$ enclosing 
$(U,\ol{U})$ $($where we put 
$U:= X \cap \ol{U})$, there exists some $\lam \in (0,1)\cap\Gamma^*$ 
such that the $\nabla$-module $E_{\cE}$ associated to 
$\cE$ is defined on $\{x \in \cP_{K} \,\vert\, 
|t(x)| \geq \lam\}$ and that the restriction of $E_{\cE}$ to 
$\cQ_{K} \times A^1_K[\lam,1)$ is $\Sigma_i$-unipotent $($where 
$\cQ$ is the zero locus of $t$ in $P$ and $i$ is any index with 
$\ol{U}\cap Z \subseteq Z_i$, which is unique 
if $\ol{U} \cap Z$ is non-empty. When $\ul{U} \cap Z$ is 
empty, we regard this last condition as vacuous one. $)$ 
\end{prop}

Next, we proceed to recall the notion of isocrystals on log convergent site 
with exponents in $\Sigma$. To do so, first we should recall the 
definition of log-$\nabla$-modules and exponents of them. 

\begin{defn}[{\cite[2.3.7]{kedlayaI}}]\label{lognabdef}
Let $\fX$ be a rigid space over $K$ 
and let $x_1, ..., x_r$ be elements in $\Gamma(\fX,\cO_{\fX})$. 
Then a log-$\nabla$-module on $\fX$ with respect to $x_1, ..., x_r$ 
is a locally free $\cO_{\fX}$-module 
$E$ endowed with an integrable $K$-linear log connection 
$\nabla: E \lra E \otimes_{\cO_{\fX}} \omega^1_{\fX}$, where 
$\omega^1_{\fX}$ denotes the sheaf of continuous log differentials 
with respect to $\dlog x_i \, (1 \leq i \leq r)$. 
\end{defn} 

When we are given $\fX, x_1,...,x_r, (E,\nabla)$ as above, 
we can define the residue $\res_i$ of $E$ along 
$\fD_i = \{x_i=0\}\,(1 \leq i \leq r)$, which is an element in 
$\End_{\cO_{\fD_i}}(E|_{\fD_i})$, in the same way as in the case of 
integrable log connection on algebraic varieties. When $\fX$ is 
smooth, zero loci of $x_i$ are affinoid smooth and meet 
transversally, there exists a minimal monic polynomial 
$P_i(x) \in K[x]$ with $P_i(\res_i)=0 \,(1 \leq i \leq r)$ by the 
argument of \cite[1.5.3]{bc}. 
We call the roots of $P_i(x)$ the exponents of $(E,\nabla)$ along 
$\fD_i$. \par 
If $\ol{U}$ is an open subscheme of $\ol{X}$ and 
$((U,\ol{U},\cP,i,j),(t_1,...,t_r))$ is 
a charted standard small frame enclosing $(U,\ol{U})$ (where 
$U := X \cap \ol{U}$), 
an isocrystal $\cE$ on log convergent site $((\ol{X},M_{\ol{X}})/O_K)_{\conv}$ 
induces a log-$\nabla$-module on $\cP_K$ with respect to 
$t_1,...,t_r$, which we denote by $E_{\cE} = 
(E_{\cE},\nabla)$. Using this construction, 
we can define the notion of isocrystals on log convergent site 
with exponents in $\Sigma$ in the following way. 

\begin{defn}\label{defsigexpo}
Let $(X,\ol{X}), M_{\ol{X}}$ be as above and 
let $\Sigma = \prod_{i=1}^r\Sigma_i$ be a 
subset of $\Z_p^r$. 
Then we say that a locally free isocrystal $\cE$ on $((\ol{X},M)/O_K)_{\conv}$ 
has exponents in $\Sigma$ if there exist an affine open covering 
$\ol{X} = \bigcup_{\alpha\in \Delta} \ol{U}_{\alpha}$ and 
charted standard small frames 
$(
(U_{\alpha},\ol{U}_{\alpha},\cP_{\alpha},i_{\alpha},j_{\alpha}), 
(t_{\alpha,1}, ..., t_{\alpha,r}))$ enclosing 
$(U_{\alpha},\ol{U}_{\alpha})$ $(\alpha \in \Delta$, where we put 
$U_{\alpha} := X \cap \ol{U}_{\alpha})$ such that, for 
any $\alpha \in \Delta$ and any $i$ $(1 \leq i \leq r)$, 
all the exponents of the log-$\nabla$-module $E_{\cE, \alpha}$ 
on $\cP_{\alpha,K}$ induced by $\cE$ along the locus $\{t_{\alpha,i}=0\}$ 
are contained in $\Sigma_i$. 
\end{defn}

Then the main result of the paper \cite{sigma} is as follows: 

\begin{thm}[{\cite[3.16]{sigma}}]\label{mainsigma}
Let $(X,\ol{X}), Z=\bigcup_{i=1}^rZ_i, M_{\ol{X}}$ be as above 
and let $\Sigma := \prod_{i=1}^r\Sigma_i$ be a subset of 
$\Z_p^r$ which is $\NID$ and $\NLD$. Then we have the canonical 
equivalence of categories 
\begin{equation}\label{maineq1}
j^{\d}: \left( 
\begin{aligned}
& \text{{\rm isocrystals on}} \\ 
& \text{$((\ol{X},M)/O_K)_{\conv}$} \\ 
& \text{{\rm with exponents in $\Sigma$}}
\end{aligned}
\right) 
\os{=}{\lra} 
\left( 
\begin{aligned}
& \text{{\rm overcongent isocrystals}} \\ 
& \text{{\rm on $(X,\ol{X})/K$ having}} \\
& \text{{\rm $\Sigma$-unipotent monodromy}} 
\end{aligned}
\right), 
\end{equation}
which is defined by the restriction. 
\end{thm}

\section{Proofs}

In this section, we give a proof of Theorem \ref{mainthm}. In the 
course of the proof, we prove a kind of cut-by-curves criteria 
on solvability, highest ramification break and exponent of $\nabla$-modules. 

\subsection{Geometric set-up}

In this subsection, we give a geometric set-up which we use in the 
following two subsections. To do this, first let us recall the notion 
of tubular neighborhood (for schemes), which is introduced 
in \cite{ab}. 

\begin{defn}[{\cite[I 1.3.1]{ab}}]
Let $\ol{X}$ be a smooth $k$-variety and let $i:Z \hra \ol{X}$ be a 
smooth divisor. Then $(f:\ol{X} \lra Y, t:\ol{X} \lra \Af^1_k)$ is 
a coodinatized tubular neighborhood of $Z$ in $X$ if it admits the 
following commutative diagram 
\begin{equation}
\label{tn}
\xymatrix{
Z \ar@{^{(}->}[r]^i \ar[d] \ar@{}[rd]|\Box & \ol{X} \ar[d]^{\phi} \\
Y \ar@{^{(}->}[r] \ar@{=}[rd] & \Af^1_Y \ar[r]^{\pr_1} 
\ar@{}[rd]|\Box 
\ar[d]^{\pr_2} & \Af^1_k \ar[d] \\ 
& Y \ar[r] & \Spec k, 
}
\end{equation}
where $Y$ is an affine connected smooth $k$-variety, 
$\phi$ is etale, two squares are Cartesian and $f=\pr_2 \circ \phi, 
t=\pr_1 \circ \phi$. 
\end{defn}

\begin{rem} 
The smoothness of $\ol{X}, Z, Y$ and the affineness, connectedness of 
$Y$ are not assumed in \cite{ab}. However, this change 
does not cause any problem 
for our purpose. Also, the above diagram looks different from that 
given in \cite[I 1.3.1]{ab}, but they are in fact equivalent. 
\end{rem}

We have the following result on the existence of 
coodinatized tubular neighborhoods. 

\begin{lem}[{\cite[I 1.3.2]{ab}}]
Let $\ol{X}$ be a smooth $k$-variety and let $Z$ be a smooth divisor on it. 
Then there exists an open covering $\ol{X} = 
\bigcup_{\alpha} \ol{X}_{\alpha}$ such that each $Z \cap \ol{X}_{\alpha} 
\hra \ol{X}_{\alpha}$ admits a structure of a 
coodinatized tubular neighborhood. 
\end{lem}

In the next subsection, we will consider the following 
geometric situation: 

\begin{hypo}\label{hyp}
Let $\ol{X}$ be a connected affine smooth $k$-variety, let 
$i:Z \hra \ol{X}$ be a non-empty connected 
smooth divisor admitting a structure of a 
coodinatized tubular neighborhood and fix a diagram \eqref{tn}. 
Then there exists a lifting of the diagram \eqref{tn} to a diagram 
of smooth formal schemes over $\Spf O_K$, 
that is, there exists a diagram 
\begin{equation}
\label{tn2}
\xymatrix{
\cZ \ar@{^{(}->}[r]^{\wh{i}} \ar[d] \ar@{}[rd]|\Box & 
\ol{\cX} \ar[d]^{\wh{\phi}} \\
\cY \ar@{^{(}->}[r] \ar@{=}[rd] & \wh{\Af}^1_{\cY} \ar[r]^{\wh{\pr}_1} 
\ar@{}[rd]|\Box 
\ar[d]^{\wh{\pr}_2} & \wh{\Af}^1_{O_K} \ar[d] \\ 
& \cY \ar[r] & \Spf O_K 
}
\end{equation}
consisting of smooth formal schemes over $\Spf O_K$ 
whose special fiber gives the diagram \eqref{tn} 
such that $\wh{\phi}$ is etale and two squares are Cartesian. 
We fix this diagram. Also, we denote the composite 
$\wh{\pr}_1 \circ \wh{\phi}$ also by $t$, by abuse of notation. 
Let us put the open immersion $X := \ol{X}-Z \hra \ol{X}$ by $j$ and 
the closed immersion $\ol{X} \hra \ol{\cX}$ by $\iota$. Also, 
we fix an inclusion $\Gamma(\cZ_K,\cO_{\cZ_K}) \hra L$, where 
$L$ is a field complete with respect to a norm which restricts to 
the supremum norm on $\cZ_K$. \par 
For a separable closed point $y$ in $Y$, 
a lift of $y$ in $\cY$ is a closed sub formal scheme 
$\ti{y} \hra \cY$ etale over $\Spf O_K$ such that 
$\ti{y} \times_{\cY} Y = y$. For a separable closed point $y$ in $Y$ and 
its lift $\ti{y}$ in $\cY$, 
denote the pull-back  by 
$\Af^1_{y} \hra \Af^1_{Y}, \Af^1_{\ti{y}} \hra \Af^1_{\cY}$ 
of the upper square of \eqref{tn}, \eqref{tn2} by 
\begin{equation}\label{tn-y}
\xymatrix{
Z_y \ar@{^{(}->}[r] \ar[d] \ar@{}[rd]|\Box & \ol{X}_y \ar[d] \\
y \ar@{^{(}->}[r] & \Af^1_y, 
}
\hspace{3cm}
\xymatrix{
\cZ_{\ti{y}} \ar@{^{(}->}[r] \ar[d] \ar@{}[rd]|\Box & \ol{\cX}_{\ti{y}} 
\ar[d] \\ \ti{y} \ar@{^{(}->}[r] & \hat{\Af}^1_{\ti{y}}, 
}
\end{equation}
respectively, and let us put $X_y := \ol{X}_y - Z_y$. 
\end{hypo}

Under Hypothesis \ref{hyp}, $((X,\ol{X},\ol{\cX},\iota,j),t)$ forms 
a charted smooth standard small frame. So, an overconvergent 
isocrystal $\cE$ on $(X,\ol{X})/K$ induces a $\nabla$-module 
on $\fU_{\lam} = \{x \in \ol{\cX}_K\,\vert\, |t(x)| \geq \lam\}$ 
for some $\lam \in (0,1)\cap\Gamma^*$, hence a $\nabla$-module 
$E_{\cE}$ on $\cZ_K \times A^1_K[\lam,1)$. 
We denote the induced $\nabla$-module on $A^1_L[\lam,1)$ by 
$E_{\cE,L}$. On the other hand, for a separable closed 
point $y \in Y$ and its 
lift $\ti{y}$ in $\cY$, 
we can restrict $\cE$ to an overconvergent isocrystal $\cE_y$ on 
$(X_y,\ol{X}_y)/K$, and by using the diagrams \eqref{tn-y}, 
we can define the $\nabla$-module $E_{\cE,\ti{y}}$ on 
$\cZ_{\ti{y},K} \times A^1_K[\lam,1)$ for some $\lam$. 
$E_{\cE,\ti{y}}$ is nothing but the restriction of $E_{\cE}$ to 
$\cZ_{\ti{y},K} \times A^1_K[\lam,1)$. \par 
Note that, in the situation of Hypothesis \ref{hyp}, we have 
a closed immersion $(\ol{X}_y, \allowbreak 
M_{\ol{X}_y}) \hra (\ol{X},M_{\ol{X}})$, 
where $M_{\ol{X}_y}, M_{\ol{X}}$ denotes the log structure on $\ol{X}_y, 
\ol{X}$ associated to $Z_y, Z$, respectively. Note also that 
$\ol{X}_y$ is a curve, $Z_y$ is a smooth divisor on it and that 
$\cZ_{\ti{y}}$ is a formal scheme etale over 
$\Spf O_K$. Hence $\cZ_{\ti{y},K} \times A^1_K[\lam,1)$ has the form 
$\coprod_{j=1}^a A_{K_j}[\lam,1)$ for some finite extensions 
$K_j \,(1 \leq j \leq a)$ of $K$. 

\subsection{Criteria on solvability, highest ramification break and exponent}

Let the notations be as in Hypothesis \ref{hyp} and fix 
$\lam \in (0,1)\cap\Gamma^*$. Let $E$ be a $\nabla$-module on 
$\cZ_K \times A^1_K[\lam,1)$ relative to $\cZ_K$ and let $E_L$ be 
the $\nabla$-module on $A^1_L[\lam,1)$ induced by $E$. On the other hand, 
for a separable closed point $y$ in $Y$ and its lift $\ti{y}$ in $\cY$, 
let $E_{\ti{y}}$ be the restriction of $E$ to $\cZ_{\ti{y},K} \times 
A^1_K[\lam,1)$. First we prove a kind of cut-by-curves criterion on 
highest ramification break, assuming certain solvability: 

\begin{thm}\label{thm1}
Let the notations be as above and assume the condition $(*)$ below$:$\\
\quad \\
$(*)$ \,\,\,\, $E_L$ is solvable and for 
any separable closed point $y \in Y$ and its lift $\ti{y}$ in $\cY$, 
$E_{\ti{y}}$ is also solvable on any connected component of 
$\cZ_{\ti{y},K} \times A^1_K[\lam,1)$. \\
\quad \\
Then the following are equivalent$:$ 
\begin{enumerate}
\item 
$E_{L}$ has highest ramification break $b$. 
\item 
There exists a dense open subset $U \subseteq Y$ such that, for any 
separable closed point $y$ in $U$ and its lift $\ti{y}$ in $\cY$, 
$E_{\ti{y}}$ has 
highest ramification break $b$ on any connected component of 
$\cZ_{\ti{y},K} \times A^1_K[\lam,1)$. 
\end{enumerate}
\end{thm}

\begin{rem}\label{solrem}
The condition $(*)$ is satisfied when $E$ is the $\nabla$-module induced 
from an overconvergent isocrystal on $(X,\ol{X})/K$, by Proposition 
\ref{ocsolv}. 
\end{rem}

\begin{rem}\label{lrem}
The solvability and the highest ramification break for $E_L$ is 
independent of the choice of $L$ in Hypothesis \ref{hyp}: 
If we denote the completion of ${\rm Frac}\,\Gamma(\cZ_K,\cO_{\cZ_K})$ 
with respect to the supremum norm by $L_0$, the inclusion 
$\Gamma(\cZ_K,\cO_{\cZ_K}) \allowbreak \hra L$ factors as 
$\Gamma(\cZ_K,\cO_{\cZ_K}) \hra L_0 \hra L$ and the norms of 
$L_0$ and $L$ are compatible. Hence, by definition, the solvability 
of $E_L$ is equivalent to that of $E_{L_0}$ and the highest 
ramification break of $E_L$ is equal to that of $E_{L_0}$. 
\end{rem}

Before giving the proof, we first prove a technical lemma. 

\begin{lem}\label{tech}
Let $K'$ be a field containing $K$ which is complete with respect to 
a norm which extends the given norm on $K$ and let 
$O_{K'}$ be the ring of integers of $K'$. 
Let $I \subseteq (0,1)$ be a closed aligned interval of positive length 
and let $a = \sum_{n \in \Z} a_nt^n$ be a non-zero element of 
$\Gamma(\cZ_{K'} \times A^1_{K'}(I),\cO)$. 
$($Here $\cZ_{K'}$ denotes the rigid analytic space over $K'$ 
associated to $\cZ \widehat{\otimes}_{O_K} O_{K'}$.$)$ 
Then there exist 
an open dense sub affine formal scheme $\cU \subseteq \cZ$ and 
a closed aligned subinterval $I' \subseteq I$ of positive length 
satisfying the following conditions$:$ 
\begin{enumerate}
\item 
$a \in \Gamma(\cU_{K'} \times A^1_{K'}(I'),\cO^{\times})$. 
\item 
For any $u \in \cU_{K'}$ and $\rho \in I'$, we have 
$|a(u)|_{\rho} = |a|_{\rho}$, where 
$a(u) := \sum_{n\in \Z} a_n(u)t^n \allowbreak 
\in \Gamma(u \times A^1_{K'}(I'),\cO)$ and 
$|\cdot|_{\rho}$ denotes the $\rho$-Gauss norm. 
\end{enumerate}
\end{lem}

\begin{proof}
In this proof, $|\cdot|$ denotes the supremum norm. Let us write 
$I = [\alpha,\beta]$. By \cite[3.1.7, 3.1.8]{kedlayaI}, we have 
$$|a| = \max ( \sup_n(|a_n|\alpha^n), \sup_n(|a_n|\beta^n)) = 
\max ( \sup_{n\leq 0}(|a_n|\alpha^n), \sup_{n\geq 0}(|a_n|\beta^n)).$$ 
Let us define finite subsets 
$A \subseteq \Z_{\leq 0}$, $B \subseteq \Z_{\geq 0}$ by 
$A := \{n \leq 0\,\vert\, |a_n|\alpha^n = |a|\}, 
B := \{n \geq 0\,\vert\, |a_n|\beta^n = |a|\}.$ Then we have 
$A \cup B \not= \emptyset$. \par 
Let us first consider the case $A \not= \emptyset$. Let $n_0$ be 
the maximal element of $A$. Then, since $a_{n_0} \not= 0$, there 
exists an element $b \in K^{\times}$ such that 
$ba_{n_0} \in \Gamma(\cZ,\cO_{\cZ})$ and that the image $\ol{ba_{n_0}}$ of 
$ba_{n_0}$ in $\Gamma(Z,\cO_Z)$ is non-zero. Let $\cU \subseteq \cZ$ be the 
open dense affine sub formal scheme such that $\ol{ba_{n_0}}$ is invertible on 
$\cU \times_{\cZ} Z$. Then we have $ba_{n_0} \in 
\Gamma(\cU,\cO_{\cU}^{\times})$. So, for all $u \in\cU_{K'}$, we have 
$|a_{n_0}(u)| = |b^{-1}|$ and hence $|a_{n_0}(u)| = |a_{n_0}|$. 
(Here note that, for elements in $\Gamma(\cZ_{K'}, \cO)$, 
its supremun norm on $\cZ_{K'}$ is the same as that on $\cU_{K'}$.) 
Next we prove the following claim: \\
\quad \\
{\bf claim.} \,\, There exists a closed aligned subinterval 
$I' \subseteq I$ of positive length 
such that $|a_n|\rho^n < |a_{n_0}|\rho^{n_0}$ 
for any $n \in \Z, \not= n_0$ and $\rho \in I'$. \\
\quad \\
Let us put $C := \{n \in \Z \,\vert\, \max(|a_n|\alpha^n, |a_n|\beta^n) 
\geq |a_{n_0}|\beta^{n_0}\}$. Then $C$ is a finite set containing $A$. 
If $n \in A, \not=n_0$, we have $|a_n|\alpha^n = |a_{n_0}|\alpha^{n_0}$ 
and $n < n_0 \leq 0$. Hence we have 
$|a_n|\rho^n < |a_{n_0}|\rho^{n_0}$ for any $\rho \in (\alpha,\beta]$. 
For $n \in C - A$, we have $|a_n|\alpha^n < |a_{n_0}|\alpha^{n_0}$. 
So there exists $\beta' \in (\alpha, \beta]$ such that, for any 
$n \in C-A$ and for any 
$\rho \in [\alpha,\beta']$, we have $|a_n|\rho^n < |a_{n_0}|\rho^{n_0}$. 
For $n \not\in C$, we have, for any $\rho \in I$, the inequalities 
$$ 
|a_n|\rho^n \leq \max(|a_n|\alpha^n,|a_n|\beta^n) < 
|a_{n_0}|\beta^{n_0} \leq |a_{n_0}|\rho^{n_0}. $$
Summing up these, we see the claim. \par 
Let us put $f := \sum_{n \not= n_0}(a_n/a_{n_0})t^{n-n_0} \in 
\Gamma(\cU_{K'} \times A^1_{K'}(I'),\cO)$ and 
take any $u\in \cU_{K'}$, 
$\rho \in I'$. Then we have 
$$ |f(u)|_{\rho} \leq 
\dfrac{\sup_{n\not=n_0}(|a_n|\rho^n)}{|a_{n_0}(u)|\rho^{n_0}} = 
\dfrac{\sup_{n\not=n_0}(|a_n|\rho^n)}{|a_{n_0}|\rho^{n_0}} <1. $$
So we have $|f| <1$. So we have 
$a = a_{n_0}t^{n_0}(1+f) \in \Gamma(\cU_{K'} \times 
A^1_{K'}(I'),\cO^{\times}).$ 
Moreover, for any $u\in \cU_K$ and $\rho \in I'$, 
we have 
$$ 
|a(u)|_{\rho} = \sup_{n}(|a_n(u)|\rho^n) = |a_{n_0}(u)|\rho^{n_0} 
= |a_{n_0}|\rho^{n_0}
= \sup_{n}(|a_n|\rho^n) = |a|_{\rho}. $$
We can prove the lemma in the case $B\not= \emptyset$ in the same way. 
(In this case, we define $n_0$ to be the minimal element of $B$.) 
So we are done. 
\end{proof}

Now we give a proof of Theorem \ref{thm1}. 

\begin{proof}[Proof of Theorem \ref{thm1}]
First we prove (2)$\,\Longrightarrow\,$(1), assuming 
(1)$\,\Longrightarrow\,$(2). Assume (2) and assume that 
$E_{L}$ has highest ramification break $b'$. Then, since we
assumed the implication (1)$\,\Longrightarrow\,$(2), 
there exists a dense open subset $U' \subseteq Y$ such that, for any 
separable closed point $y$ in $U'$ and its lift $\ti{y}$ in $\cY$, 
$E_{\ti{y}}$ has 
highest ramification break $b'$. Since $U \cap U'$ contains 
a separable closed point, this implies the equality $b=b'$ as desired. 
So it suffices to prove (1)$\,\Longrightarrow\,$(2). We will prove this. \par 
For a separable closed point $x$ in $Z$, we define 
a lift of $x$ in $\cZ$ 
as a closed sub formal scheme $\ti{x} \hra \cZ$ etale over 
$\Spf O_K$ such that $x = \ti{x} \times_{\cZ} Z$. Then, to prove the 
assertion, it suffices to prove that there exists an open dense subscheme 
$V \subseteq Z$ such that, for any separable closed point $x$ and its lift 
$\ti{x}$ in $\cZ$, the restriction $E_{\ti{x}}$ of $E$ to 
$\ti{x}_K \times A^1_K[\lam,1)$ has highest ramification break $b$. 
Indeed, if this is true, we have the assertion if we put 
$U := Y - \ol{\phi (Z-V)}$. So we will prove this claim. \par 
First we prove the above claim in the case 
$b=0$. We may assume that $R(E_{\cE,L},\rho)=\rho$ for any 
$\rho \in [\lam,1)$. 
Take any closed aligned subinterval $I \subseteq [\lam,1)$ 
of positive length and 
put $A := \Gamma(\cZ_K \times A_K(I),\cO)$, $\E := 
\Gamma(\cZ_K \times A_K(I),E)$. Then $A$ is an integral domain and 
$\E$ is a finitely generated $A$-module. Let 
$\e := (\e_1,...,\e_{\mu})$ be a basis of $\Frac A \otimes_A \E$ 
as $\Frac A$-vector space and let $(\f_1, ..., \f_{\nu})$ be a set 
of generator of $\E$ as $A$-module. Then there exist 
$b_{ij} := b'_{ij}/b''_{ij}, c_{ji} := c'_{ji}/c''_{ji} \in \Frac A 
\,(1 \leq i \leq \mu, 1 \leq j \leq \nu)$ such that 
$\e_i = \sum_{j=1}^{\nu} b_{ij} \f_j \,(\forall i), 
\f_j = \sum_{i=1}^{\mu} c_{ji} \e_i \,(\forall j)$. By 
Lemma \ref{tech}, there exist an open dense sub affine formal scheme 
$\cV \subseteq \cZ$ and closed aligned subinterval $I' \subseteq I$ 
of positive length 
such that $b''_{ij}, c''_{ji} \in \Gamma(\cV_K \times A^1_K(I'),
\cO^{\times})$. Then we see that $\e$ forms a basis of 
$\Gamma(\cV_K \times A^1_K(I'),E)$ as 
$\Gamma(\cV_K \times A^1_K(I'),\cO)$-module. 
Then let us put $V := \cV \times_{\cZ} Z$ and 
let $\pa: E \lra E$ 
be a morphism on $\cV_K \times A^1_K(I')$ defined as the composite 
$$ E \os{\nabla}{\lra} E \otimes 
\Omega^1_{\cV_K \times A^1_K(I')/\cV_K} = 
E dt \os{=}{\lra} E, $$
where the last 
map is defined by $xdt \mapsto x$. For $n \in \N$, let $G_n$ be 
the matrix expression of $\pa^n$ with respect to the basis $\e$ 
(so we have $G_n \in \Mat_{\mu}(\Gamma(\cV_K \times A^1_K(I'),\cO)) 
\subseteq \Mat_{\mu}(\Gamma(A^1_L(I'),\cO))$). \par 
For any separable closed point $x \in V$ and its lift $\ti{x}$ in $\cZ$, 
$E$ restricts to a $\nabla$-module $E_{\ti{x}}$ on 
$\ti{x}_K \times A^1_K(I')$, $\pa$ restricts to the morphism 
$\pa_{\ti{x}}:E_{\ti{x}} \lra E_{\ti{x}}$ induced from the 
$\nabla$-module structure of $E_{\ti{x}}$ and the matrix 
$G(\ti{x}_K) \in \Mat_{\mu}(\Gamma(\ti{x}_K \times A^1_K(I'),\cO))$ 
gives the matrix expression of $\pa_{\ti{x}}$ with respect to the basis 
$\e$. By definition of the supremum norm and \cite[3.1.7, 3.1.8]{kedlayaI}, 
we have the inequality $|G_n(\ti{x}_K)|_{\rho} \leq |G_n|_{\rho}$ for any 
$\rho \in I' \cap \Gamma^*$. Hence we have 
\begin{align*}
\rho \geq R(E_{\ti{x}},\rho) & = \min (\rho, 
\varliminf_{n\to\infty}|G_n(\ti{x}_K)/n!|_{\rho}^{-1/n}) \\ 
& \geq 
\min (\rho, 
\varliminf_{n\to\infty}|G_n/n!|_{\rho}^{-1/n}) \os{(*)}{=} 
R(E_{L},\rho) = \rho, 
\end{align*}
that is, $R(E_{\ti{x}},\rho)=\rho$. (As for $(*)$, note that 
the map $\Gamma(\cZ_K \times A_K^1(I), \cO) \lra L(t)_{\rho}$ factors 
through $\Gamma(\cV_K \times A^1_K(I'),\cO)$.) 
Since this is true for any 
$\rho \in I' \cap \Gamma^*$, we see that the highest ramification 
break of $E_{\ti{x}}$ is equal to $0$, by Corollary \ref{elementary} (2). 
So we have proved the desired claim in the case 
$b=0$. \par 
Next we treat the case $b>0$. 
(The proof in this case partly follows 
the argument in \cite[1.3.1]{kedlayaswanII}.) 
Let $\zeta$ be a primitive $p$-th root of unity. Let us put 
$K' := K(\zeta)$, let $O_{K'}$ be the ring of integers of $K'$ and 
let $\cZ_{K'}$ be the rigid space over $K'$ associated to 
$\cZ \widehat{\otimes}_{O_K} O_{K'}$. Also, let us take an 
inclusion $\Gamma(\cZ_{K'},\cO) \hra L'$, where $L'$ is a 
field complete with respect to a norm which restricts to 
the supremum norm on $\cZ_{K'}$. Since the supremum norm on 
$\cZ_{K'}$ is compatible with that on $\cZ_K$, the restriction 
of the norm on $L'$ to $\Gamma(\cZ_K,\cO)$ gives the supremum norm 
on $\cZ_K$. Hence, by Remark \ref{lrem}, the $\nabla$-module $E_{L'}$ 
on $A_{L'}[\lam,1)$ induced by $E$ has solvable and it has 
highest ramification break $b$. So we may assume that 
$R(E_{L'},\rho)=\rho^{b+1}$ for any 
$\rho \in [\lam,1)$. Then there exists a closed aligned 
subinterval $I \subseteq [\lam,1)$ of positive length such that 
there exists a positive integer $m$ with 
$$ p^{-1/p^{m-1}(p-1)}\rho < R(E_{L'},\rho) < p^{-1/p^m(p-1)}\rho \,\,\,\, 
(\forall \rho \in I). $$
Let $F$ be a $\nabla$-module on $\cZ_{K'} \times A^1_{K'}(I^{p^m})$ 
with $\varphi^{m*}F = E_{K'}$ (where $E_{K'}$ denotes 
the restriction of $E$ to 
$\cZ_{K'} \times A^1_{K'}(I)$ and $\varphi$ denotes the morphism 
$\cZ_{K'} \times A^1_{K'}(I) \lra \cZ_{K'} \times A^1_{K'}(I^p)$ over 
$\cZ_{K'}$ defined by $t \mapsto t^p$) 
such that the induced $\nabla$-module $F_{L'}$ on $A^1_{L'}(I^{p^m})$ is 
the $m$-fold Frobenius antecedent of $E_{L'}$. 
(The existence of such $F$ is assured by Proposition \ref{frobantrel}.) 
Then we have 
$R(F_{L'},\rho^{p^m}) = R(E_{L'},\rho)^{p^m}$. Hence we have the inequality 
$p^{-p/(p-1)} \rho^{p^m} 
< R(F_{L'},\rho^{p^m}) < p^{-1/(p-1)}\rho^{p^m} \, (\rho \in I)$. 
This implies the inequality 
\begin{equation}\label{ineq}
|\pa|_{F_{L',\rho^{p^m}},\sp} > \rho^{-p^m} = 
|\pa|_{L'(t)_{\rho^{p^m}}}. 
\end{equation}
Let us  
put $A := \Gamma(\cZ_{K'} \times A_{K'}^1(I^{p^m}),\cO)$, $\F := 
\Gamma(\cZ_{K'} \times A_{K'}(I),F)$. Then we can define 
$\pa:F \lra F$ on $\cZ_{K'} \times A_{K'}^1(I^{p^m})$ as before and 
it induces the map $\pa:\F \lra \F$. 
Let 
$\v$ be a cyclic vector of $\Frac A \otimes_A \F$, put 
$\pa^{\mu}(\v) = \sum_{i=0}^{\mu-1}a_i\pa^1(\v)$ (where $\mu$ is the rank of 
$F$) with $a_i = a'_i/a''_i \in \Frac A$ 
and let $(\f_1, ..., \f_{\nu})$ be a set 
of generators of $\F$ as $A$-module. Then there exist 
$b_{ij} := b'_{ij}/b''_{ij}, c_{ji} := c'_{ji}/c''_{ji} \in \Frac A 
\,(0 \leq i \leq \mu -1, 1 \allowbreak \leq j \leq \nu)$ such that 
$\pa^i(\v) = \sum_{j=1}^{\nu} b_{ij} \f_j \,(\forall i), 
\f_j = \sum_{i=0}^{\mu -1} c_{ji} \pa^i (\v) \,(\forall j)$. By 
Lemma \ref{tech}, there exists an open dense sub affine formal scheme 
$\cV \subseteq \cZ$ and a closed aligned subinterval $I' \subseteq I$ 
of positive length such that $a''_i, 
b''_{ij}, c''_{ji} \in \Gamma(\cV_{K'} \times A^1_{K'}({I'}^{p^m}),
\cO^{\times})$ (where $\cV_{K'}$ is defined in the same way as $\cZ_{K'}$, by 
using $\cV$ instead of $\cZ$). 
Then we see that $\v,\pa(\v),...,\pa^{\mu-1}(\v)$ 
form a basis of 
$\Gamma(\cV_{K'} \times A^1_{K'}({I'}^{p^m}),F)$ as 
$\Gamma(\cV_{K'} \times A^1_{K'}({I'}^{p^m}),\cO)$-module. 
By shrinking $\cV$ and $I'$ further and using Lemma \ref{tech} again, 
we may assume that 
$|a_i(u)|_{\rho^{p^m}} = |a_i|_{\rho^{p^m}}$ for any $u \in 
\cV_{K'}$ and $\rho \in I'$. \par 
Now let us put $V := \cV \times_{\cZ} Z$, 
take any separable closed point $x \in V$, its lift $\ti{x}$ in 
$\cV$, a point $\ti{x}_{K'} \in \cV_{K'}$ which lies above 
$\ti{x}_K \in \cV_K$ and any $\rho \in I'$. 
Let $r$ be the least slope of the lower convex hull of 
the set $\{(-i,\log |a_i|_{\rho^{p^m}}) \,\vert\, 0 \leq i \leq \mu -1\} = 
\{(-i,\log |a_i(\ti{x}_{K'})|_{\rho^{p^m}}) \,\vert\, 0 \leq i \leq \mu -1 \}
\subseteq \R^2$. Then we have 
$$ \max(|\pa|_{F_{L',\rho^{p^m}},\sp}, 
|\pa|_{L'(t)_{\rho^{p^m}}}) = 
\max(|\pa|_{F_{\ti{x},\rho^{p^m}},\sp}, 
|\pa|_{K'(t)_{\rho^{p^m}}}) = e^{-r} $$ 
(where $F_{\ti{x}}$ is the $\nabla$-module on 
$\ti{x}_{K'} \times A^1_{K'}({I'}^{p^m})$ induced by $F$) 
and by \eqref{ineq}, this implies the equality  
$|\pa|_{F_{L',\rho^{p^m}},\sp} = e^{-r} = 
|\pa|_{F_{\ti{x},\rho^{p^m}},\sp}$. Hence we have 
$R(F_{L'},\rho^{p^m}) = R(F_{\ti{x}},\rho^{p^m})$. 
In particular, we have  
$R(F_{\ti{x}},\rho^{p^m}) > p^{-p/(p-1)}\rho^{p^m}$. 
Hence $F_{\ti{x}}$ is the $m$-fold Frobenius antecedent of 
${\varphi^m}^*F_{\ti{x}} = E_{\ti{x},K'}$, where 
$E_{\ti{x},K'}$ is the restriction of $E_{\ti{x}}$ to 
$\ti{x}_{K'} \times A^1_{K'}(I')$. Noting the equality 
$R(E_{\ti{x}},\rho) = R(E_{\ti{x},K'},\rho)$, 
we obtain the equalities 
$$ 
R(E_{\ti{x}},\rho) = 
R(E_{\ti{x},K'},\rho) = 
R(F_{\ti{x}},\rho^{p^m})^{p^{-m}} = 
R(F_{L'},\rho^{p^m})^{p^{-m}} = R(E_{L'},\rho) = \rho^{b+1}. $$ 
Since this is true for any $\rho \in I'$, we can conclude that 
the highest ramification break of $E_{\ti{x}}$ is $b$ by 
Corollary \ref{elementary} (2). So we 
have finished the proof of the theorem. 
\end{proof}

Even if we do not assume the solvability assumption $(*)$, we still 
have the following criterion for solvability and highest ramification 
break: 

\begin{thm}\label{thm1bis}
Let the notations be as in the beginning of this subsection and 
assume that $k$ is uncountable. 
Then the following are equivalent$:$ 
\begin{enumerate}
\item 
$E_{L}$ is solvable with highest ramification break $b$ 
$($resp. not solvable$)$. 
\item 
There exists a decreasing sequence of dense open subsets 
$\{U_i\}_{i\in \N}$ in $Y$ such that, for any separable 
closed point $y$ in $\bigcap_{i\in \N}U_i$ and its lift $\ti{y}$ in 
$\cY$, $E_{\ti{y}}$ has 
highest ramification break $b$ 
$($resp. not solvable$)$ on any connected component of 
$\cZ_{\ti{y},K} \times A^1_K[\lam,1)$. 
\end{enumerate}
\end{thm}

Until the end of this subsection, we will sometimes 
say by abuse of terminology 
that $E_L$ or $E_{\ti{y}}$ has highest ramification break 
$+\infty$ if it is not solvable. 

\begin{proof}
The proof is similar to the proof of Theorem \ref{thm1}. 
First we prove (2)$\,\Longrightarrow\,$(1), assuming 
(1)$\,\Longrightarrow\,$(2). Assume (2) and assume that 
$E_{L}$ has highest ramification break $b'$. Then, since we
assumed the implication (1)$\,\Longrightarrow\,$(2), 
there exists a decreasing sequence $\{U'_i\}_{i \in \N}$ 
of dense open subsets in $Y$ such that, for any 
separable closed point $y$ in $\bigcap_{i\in \N}U'$ 
and its lift $\ti{y}$ in $\cY$, $E_{\ti{y}}$ has 
highest ramification break $b'$. Since $\bigcap_{i\in \N}U \cap 
\bigcap_{i\in \N}U'$ contains 
a separable closed point (because $k$ is uncountable), 
this implies the equality $b=b'$ as desired. 
So it suffices to prove (1)$\,\Longrightarrow\,$(2). \par 
Moreover, 
to prove the assertion, 
it suffices to prove that there exists a decreasing 
sequence $\{V_i\}_{i\in \N}$ of open dense subschemes in 
$Z$ such that, for any separable closed point $x$ and its lift 
$\ti{x}$ in $\cZ$, the restriction $E_{\ti{x}}$ of $E$ to 
$\ti{x}_K \times A^1_K[\lam,1)$ has highest ramification break $b$. 
Indeed, if this is true, we have the assertion if we put 
$U_i := Y - \ol{\phi (Z-V_i)}$. So we will prove this claim. \par 
Let us take a closed aligned intervals $I_m = [\alpha_m,\beta_m] 
\subseteq [\lam,1) \, (m \in \N)$ such that $\alpha_m < \beta_m < 
\alpha_{m+1}\,(\forall m)$ and that $\dlim_{m\to\infty}\alpha_m=1$. 
Then, by the argument of the proof of Theorem \ref{thm1} (in the case $b=0$), 
we see that, for each $m$, there exists a closed aligned subinterval 
$I'_m \subseteq I_m$ of positive length 
and an open dense sub affine formal scheme $\cV'_m \subseteq \cZ$ 
such that $E$ admits a free basis $\e := (\e_1,...,\e_{\mu})$ on 
$\cV'_{m,K} \times A^1_K(I'_m)$. Then, for each $m$, we can define 
the matrices $G_{m,n} \in \Mat_{\mu}(\Gamma(\cV'_{m,K} \times A^1_K(I'_m))) \, 
(n\in \N)$ as in the proof of Theorem \ref{thm1}, by using the basis $\e$. 
Let us put $V'_m := \cV'_m \times_{\cZ} Z$. Then there exists a decreasing 
sequence $\{V'_{m,n}\}_{n\in\N}$ of dense open subschemes in $V'_m$ and 
a decreasing sequence $\{I'_{m,n}\}_{n\in \N}$ 
of closed aligned subintervals of $I'_m$ of positive length 
such that, for any 
separable closed point $x$ in $V'_{m,n}$, any lift $\ti{x}$ of $x$ in 
$\cZ$, any $n' \leq n$ and any $\rho \in I'_{m,n}$, 
we have the equality $|G_{m,n'}(\ti{x}_K)|_{\rho}=|G_{m,n'}|_{\rho}$. 
Now, for each $m$, fix $\rho_m \in \bigcap_{n\in\N}I'_{m,n}$ and 
for $i\in \N$, let us put $V_i:=\bigcap_{m+n\leq i}V_i$. 
Then, for any separable closed point $x$ in $\bigcap_{i\in\N}V_i$ and 
any lift $\ti{x}$ of $x$ in $\cZ$, we have the equality 
$|G_{m,n}(\ti{x}_K)|_{\rho_m} = |G_{m,n}|_{\rho_m}$ 
for any $m,n$. Hence we have 
\begin{align*}
R(E_{\ti{x}},\rho_m) & = \min (\rho_m, 
\varliminf_{n\to\infty}|G_{m,n}(\ti{x}_K)/n!|_{\rho_m}^{-1/n}) \\ 
& = 
\min (\rho_m, 
\varliminf_{n\to\infty}|G_{m,n}/n!|_{\rho_m}^{-1/n}) = 
R(E_{L},\rho_m) 
\end{align*}
for any $m$. If $E_L$ has highest ramification break $b$ ($0\leq b<+\infty$), 
we have $R(E_{L},\rho_m)\allowbreak
=\rho_m^{b+1}$ for $m$ sufficiently large, by 
Corollary \ref{elementary} (1). So we have 
$R(E_{\ti{x}},\rho_m)=\rho_m^{b+1}$ for $m$ sufficiently large and hence 
$E_{\ti{x}}$ has also highest ramification break $b$, again by 
Corollary \ref{elementary} (1). If $E_L$ is not solvable, 
we have $\displaystyle\varlimsup_{m\to\infty} R(E_{L},\rho_m) < 1$ by 
Corollary \ref{elementary} (3). So we have 
$\displaystyle\varlimsup_{m\to\infty} R(E_{\ti{x}},\rho_m) < 1$ and hence 
$E_{\ti{x}}$ is not solvable either, again by 
Corollary \ref{elementary} (3). Hence 
we have proved the desired assertion 
and so the proof of the theorem is finished. 
\end{proof}

Next we prove a kind of cut-by-curves criterion on exponent. 

\begin{thm}\label{thm2}
Let the situation be as in the beginning of this 
subsection and assume moreover 
that $k$ is uncountable. Then the following are equivalent$:$ 
\begin{enumerate}
\item 
$E_{L}$ is solvable with 
highest ramification break $0$ and $\Exp(E_{L})= C$. 
\item 
There exists a decreasing sequence of dense open subsets 
$\{U_i\}_{i \in \N}$ in $Y$ such that, for any 
separable closed point $y$ in $\bigcap_{i \in \N} U_i$ 
and its lift $\ti{y}$ in $\cY$, 
$E_{\ti{y}}$ is solvable with 
highest ramification break $0$ and $\Exp(E_{\ti{y}})= C$ 
$($on any connected component of 
$\cZ_{\ti{y},K} \times A^1_K[\lam,1).)$ 
\end{enumerate}
\end{thm}

\begin{proof}
The strategy of the proof is similar to that of Theorems \ref{thm1} and 
\ref{thm1bis}. 
First we prove (2)$\,\Longrightarrow\,$(1), assuming 
(1)$\,\Longrightarrow\,$(2). Assume (2) and assume that 
$E_{L}$ has highest ramification break $b'$ (possibly $+\infty$). 
Then, by Theorem \ref{thm1bis}, 
there exists a decreasing sequence $\{U'_i\}_{i\in \N}$ of 
dense open subsets in $Y$ such that, for any 
separable closed point $y$ in $\bigcap_{i\in\N}U'_i$ and 
its lift $\ti{y}$ in $\cY$, 
$E_{\ti{y}}$ has 
highest ramification break $b'$. 
Since $\bigcap_{i \in \N}U_i \cap \bigcap_{i\in\N}U'_i$ 
contains a separable closed point (because $k$ is uncountable), 
this implies the equality $b'=0$. 
Then, since we
assumed the implication (1)$\,\Longrightarrow\,$(2), 
there exists a decreasing sequence $\{U''_i\}_{i\in\N}$ 
of dense open subsets in $Y$ 
such that, for any 
separable closed point $y$ in $\bigcap_{i \in \N}U''_i$ and its lift 
$\ti{y}$ in $\cY$, 
$\Exp(E_{\ti{y}})=\Exp(E_{L})$. Since $\bigcap_{i \in \N}U_i \cap 
\bigcap_{i \in \N}U''_i$ contains 
a separable closed point, this implies the equality $\Exp(E_{L}) = C$, 
as desired. 
So it suffices to prove (1)$\,\Longrightarrow\,$(2). Moreover, 
it suffices to prove that there exists a sequence of open dense subschemes 
$V_i \subseteq Z \,(i \in \N)$ 
such that, for any separable closed point $x$ in $\bigcap_{i\in \N}V_i$ 
and its lift 
$\ti{x}$ in $\cZ$, the restriction $E_{\ti{x}}$ of $E$ to 
$\ti{x}_K \times A^1_K[\lam,1)$ has highest ramification break $0$ 
with $\Exp(E_{\ti{x}})=C$. (Indeed, if this is true, 
we have the assertion if we put 
$U_i := Y - \ol{\phi (Z-V_i)}$.) So we will prove this claim. \par 
Let $K_{\infty}$ be the $p$-adic completion of $K(\mu_{p^{\infty}})$ and 
let $O_{K_{\infty}}$ be the valuation ring of $K_{\infty}$. 
Take a closed aligned interval $I \subset [\lam,1)$ of positive 
length such that 
$R(E_L,\rho)=\rho$ for $\rho \in I$. By 
the argument of the proof of Theorem \ref{thm1}, we have the following: 
If we shrink $\cZ$ and $I$ if necessary, $E$ admits a free basis 
$\e := (\e_1,...,\e_{\mu})$ on $\cZ_K \times A^1_K(I)$ and for 
any separable closed point $x$ in $Z$ and its lift $\ti{x}$ in $\cZ$, 
we have the equality $R(E_{\ti{x}},\rho)=\rho$ for all $\rho\in I$. 
In particular, 
for $\Delta_h \in (\Z/p^h\Z)^{\mu}$, the matrix functions 
$Y_{\e}(x,y)$ and $S_{h,\Delta_h}(x)$ 
associated to $E_{L}$ constructed 
in Subsection 1.4 is in fact defined as 
the matrix functions with coefficients in 
$\Gamma(\cZ_{K_{\infty}},\cO)$, where $\cZ_{K_{\infty}}$ denotes the 
rigid space over $K_{\infty}$ associated to $\cZ_{O_{K_{\infty}}} := 
\cZ \wh{\otimes}_{O_K} O_{K_{\infty}}$. 
In particular, 
we have $\det S_{h,\Delta_h}(x) \in \Gamma(\cZ_{K_{\infty}} \times 
A^1_{K_{\infty}}(I),\cO)$. 
By Lemma \ref{tech}, there exists a decreasing 
sequence of open dense sub affine 
formal schemes $\cV'_i \subseteq \cZ \,(i \in \N)$ 
and a decreasing sequence of aligned closed subintervals 
$I_i \subseteq I \,(i \in \N)$ such that, on any $u \in 
\cV'_{i,K_{\infty}}$, $\rho \in I_i$ and any $h \leq i, \Delta_h \in 
(\Z/p^h\Z)^{\mu}$, we have $|\det S_{h,\Delta_h}(u)|_{\rho} = 
|\det S_{h,\Delta_h}|_{\rho}$. Let us put $V'_i := \cV'_i \times_{\cZ} Z$ and 
let $\rho_0$ be any element of 
$\bigcap_{i \in \N}I_i$. Then, by definition of $\Exp(E_{L})$, 
there exists an element $\Delta \in \Exp(E_{L})=C$
such that, if we put $\Delta_h := \Delta \,{\rm mod}\, 
p^h$, we have the inequalities $|\det (S_{h,\Delta_{h}})|_{\rho_0} 
\leq |\det (S_{h+1,\Delta_{h+1}})|_{\rho_0}$ 
for any $h \in \N$. Then, for any separable closed point 
$x \in \bigcap_{i\in\N}V'_i$, its lift $\ti{x}$ in $\cZ$ 
and any point $\ti{x}_{K_{\infty}}$ of 
$\cZ_{K_{\infty}}$ lying above $\ti{x}_K \in \cZ_K$, we have 
$|\det (S_{h,\Delta_{h}}(\ti{x}_{K_{\infty}}))|_{\rho_0} 
\leq |\det (S_{h+1,\Delta_{h+1}}(\ti{x}_{K_{\infty}}))|_{\rho_0}$ 
for any $h \in \N$,. 
Also, by Theorem \ref{thm1bis}, there exists a decreasing sequence 
of open dense subschemes $V''_i \subseteq Z$ such that, for any 
separable closed point $x$ in $\bigcap_{i\in\N}V''_i$ and its lift 
$\ti{x}$ in $\cZ$, $E_{\ti{x}}$ is solvable with highest ramification 
break $0$. Let us put $V_i := V'_i \cap V''_i\,(i\in\N)$. 
Then, for any separable closed point $x$ in $\bigcap_{i\in\N}V_i$ and 
its lift $\ti{x}$ in $\cZ$, $E_{\ti{x}}$ is solvable with highest ramification 
break $0$ and its exponent $\Exp(E_{\ti{x}})$ 
as $\nabla$-module on $\ti{x}_K \times A^1_K(I)$ is the class of 
$\Delta$, which is equal to $C$. (See the definition of the exponent given 
in Subsection 1.4.) Hence 
$\{V_i\}_{i\in\N}$ satisfies the condition we need. 
\end{proof}

\begin{rem}
If one is interested only in the proof of Theorem \ref{mainthm}, 
Theorem \ref{thm1} is actually unnecessary 
(although we need some arguments in the proof there). 
We included it because we think that it is 
interesting itself and because 
we think that (some generalization of) this result would be useful 
in the study of overconvergent isocrystals. 
\end{rem}

\subsection{Proof of the main theorem and a variant}

In this subsection, we give a proof of Theorem \ref{mainthm} by 
using the results in the previous subsection, and give also a variant 
of Theorem \ref{mainthm} as a corollary, which treats the case 
where $k$ is not necessarily uncountable. First we give a proof 
of Theorem \ref{mainthm}: 

\begin{proof}[Proof of Theorem \ref{mainthm}]
First we prove (1)$\,\Longrightarrow\,$(2). Since $\cE$ is $\Sigma$-unipotent 
by assumption, there exists an isocrystal $\ol{\cE}$ on 
$((\ol{X},M_{\ol{X}})/O_K)_{\conv}$ with exponents in $\Sigma$ which extends 
$\cE$, by Theorem \ref{mainsigma}. 
Then, $\iota^*\ol{\cE}$ gives an isocrystal $\ol{\cE}$ on 
$((\ol{C},M_{\ol{C}})/O_K)_{\conv}$ with exponents in $\iota^*\Sigma$ 
which extends $\iota^*\cE$. Hence $\iota^*\cE$ is $\iota^*\Sigma$-unipotent
 again by Theorem \ref{mainsigma}. \par 
Next we prove (2)$\,\Longrightarrow\,$(1). By definition of 
$\Sigma$-unipotence given in Definition \ref{defsigmon}, 
it suffices to prove 
the assertion locally on $\ol{X}-Z_{\sing}$. 
Hence we may assume that $\ol{X}$ is connected and affine, 
$Z$ is a non-empty connected smooth divisor in $\ol{X}$ and that 
$Z \hra \ol{X}$ admits a coordinatized tubular neighborhood, that is, 
we may assume that we are in the situation of Hypothesis \ref{hyp}. 
Then, for each separable closed point $y$ in $Y$, $\cE$ restricts to 
the overconvergent isocrystal $\cE_y$ on 
$(X_y, \ol{X}_y)$ having $\Sigma$-unipotent monodromy. 
Let $\ti{y}$ be any lift of $y$ in $\cY$. 
By Propositions 
\ref{ocsolv}, \ref{cmunip} and \ref{wdsigmon}, 
the associated $\nabla$-module $E_{\cE,\ti{y}}$ on 
$\cZ_{\ti{y},K} \times A^1_K[\lam,1)$ is solvable with highest ramification 
$0$ such that $\Exp(E_{\cE,\ti{y}}) \in \ol{\Sigma}$. 
Then, by Theorem \ref{thm2}, the $\nabla$-module $E_{\cE,L}$ on 
$A^1_L[\lam,1)$ associated to $\cE$ is solvable with highest ramification 
break $0$ such that $\Exp(E_{\cE,L}) \in \ol{\Sigma}$. By 
Propositions \ref{generization} and \ref{cmunip}, it implies that 
the $\nabla$-module $E_{\cE}$ on $\cZ_K \times A^1_K[\lam,1)$ 
associated to $\cE$ is $\Sigma$-unipotent for some $\lam$. Hence 
$\cE$ is $\Sigma$-unipotent, as desired. 
\end{proof}

Next we introduce several notations which we need to describe 
a variant of Theorem \ref{mainthm}. Let $K,O_K,k$ be as in Convention 
(with $k$ not necessarily uncountable) and let $X, \ol{X}, 
Z = \ol{X} - X = \bigcup_{i=1}^r Z_i, M_{\ol{X}}, 
\Sigma = \prod_{i=1}^r \Sigma_i$ be as in 
Introduction. For a field extension $k \subseteq 
k'$, let us put $O_{K'} := O_K \otimes_{W(k)} W(k')$, $K' := 
{\rm Frac}\,O_K$, $X' := X \otimes_k k', \ol{X}' := \ol{X} \otimes_k k'$, 
$Z' := Z \otimes_k k' = \bigcup_{i=1}^{r'} Z'_i$ and let $M_{\ol{X}'}$ be
 the log structure on $\ol{X}'$ associated to $Z'$ and denote the projections 
$(\ol{X}',M_{\ol{X}'}) \lra (\ol{X},M_{\ol{X}}), (X',\ol{X}') \lra 
(X,\ol{X})$ by $\pi$. Then 
$\pi$ induces a well-defined morphism 
of sets $\{1,...,r\} \lra \{1,...,r'\}$ (which we denote also by 
$\pi$) by the rule $\pi(Z'_{i}) \subseteq Z_{\pi(i)}$. 
Then we put $\Sigma' := \prod_{i=1}^{r'} \Sigma_{\pi(i)} 
\subseteq \Z_p^{r'}$. \par 
Assume we are given an open immersion of smooth $k'$-curves $C' \hra \ol{C}'$ 
such that $P' := \ol{C}' - C' = \coprod_{i=1}^s P'_i$ is a simple normal 
crossing divisor, and denote the log structure on $\ol{C}'$ assciated to 
$P'$ by $M_{\ol{C}}$. Assume also that we are given a commutative diagram 
\begin{equation}\label{diagcor}
\begin{CD}
(\ol{C}',M_{\ol{C}'}) @>{\iota}>> (\ol{X},M_{\ol{X}}) \\ 
@VVV @VVV \\ 
\Spec k' @>>> \Spec k 
\end{CD}
\end{equation}
such that $\iota$ induces an exact locally closed immersion 
$\iota': (\ol{C}',M_{\ol{C}'}) \hra (\ol{X}',M_{\ol{X}'})$. Then we can 
define $\iota^*\Sigma$ by $\iota^*\Sigma := {\iota'}^*\Sigma'$, where 
${\iota'}^*\Sigma'$ is `the pull-back of $\Sigma'$ by the exact locally 
closed immersion $\iota'$' defined in Introduction. Also, $\iota$ 
defines a morphism of pairs $(C',\ol{C}') \lra (X,\ol{X})$ (which we also 
denote by $\iota$) and hence, for an overconvergent isocrystal $\cE$ on 
$(X,\ol{X})/K$, we can define the pull-back $\iota^*\cE$, which is 
an overconvergent isocrystal on $(C',\ol{C}')/K'$. With this notation, 
we can state a variant of our main theorem as follows: 

\begin{cor}\label{maincor}
Let $K,k,X,\ol{X},M_{\ol{X}},\Sigma$ be as above $(k$ is not necessarily 
uncount\-able$)$. Then, for an overconvergent isocrystal $\cE$ on 
$(X,\ol{X})/K$, the following are equivalent$:$ 
\begin{enumerate} 
\item $\cE$ has $\Sigma$-unipotent monodromy. 
\item For any field extension $k \subseteq k'$, 
for any $(C',\ol{C}'), M_{\ol{C}'}$ over $k'$ as above and for any 
diagram \eqref{diagcor} as above, $\iota^*\cE$ has $\iota^*\Sigma$-unipotent 
monodromy. 
\end{enumerate}
\end{cor}

\begin{proof}
The proof of the implication (1)$\,\Longrightarrow\,$(2) is the same as 
the proof in Theorem \ref{mainthm}. So we omit it and only give a proof of 
(2)$\,\Longrightarrow\,$(1). Fix a field extension $k \hra k'$ such that 
$k'$ is uncountable and define $X',\ol{X}', Z' = \bigcup_{i=1}^{r'}Z'_i, 
M_{\ol{X}'}, \pi:(X',\ol{X}') \lra (X,\ol{X})$ 
as above. Then, for $C',\ol{C}',M_{\ol{C}'}$ over $k'$ as above 
and for any exact locally closed immersion 
$\iota':(\ol{C}',M_{\ol{C}'}) \lra (\ol{X}', M_{\ol{X}'})$, 
$\iota^*\cE = {\iota'}^*\pi^*\cE$ has $\iota^*\Sigma = 
{\iota'}^*\Sigma'$-unipotent monodromy (with $\Sigma'$ defined as above). Hence, by 
Theorem \ref{mainthm}, $\pi^*\cE$ has $\Sigma'$-unipotent monodromy. 
So, to prove the corollary, it suffices to prove the following claim: \\
\quad \\
{\bf claim.} \,\,\, With the above notation, if $\pi^*\cE$ has 
$\Sigma'$-unipotent monodromy, $\cE$ has $\Sigma$-unipotent 
monodromy. \\
\quad \\
We prove the claim. Let us take 
an affine open covering 
$\ol{X}-Z_{\sing} = \bigcup_{\alpha\in \Delta} \ol{U}_{\alpha}$, 
put $U_{\alpha} := X \cap \ol{U}_{\alpha}$ and take 
charted smooth standard small frames 
$(
(U_{\alpha},\ol{U}_{\alpha},\cP_{\alpha},i_{\alpha},j_{\alpha}), 
t_{\alpha})$ enclosing 
$(U_{\alpha},\ol{U}_{\alpha})$ $(\alpha \in \Delta)$. 
Let us put $\Delta' := \{\alpha \in \Delta \,\vert\, 
Z \cap \ol{U}_{\alpha} \not= \emptyset\}$. For $\alpha \in \Delta'$, 
let $\cQ_{\alpha} \subseteq \cP_{\alpha}$ be the zero locus of 
$t_{\alpha}$, denote the $\nabla$-module on 
$\cQ_{\alpha,K} \times A^1_K[\lam,1)$ (for some $\lam$) induced by $\cE$ 
by $E_{\cE,\alpha}$ and let $a_{\alpha}$ be the unique index satisfying 
$\ol{U}_{\alpha} \cap Z \subseteq Z_{a_{\alpha}}$. Let 
$(
(U'_{\alpha},\ol{U}'_{\alpha},\cP'_{\alpha},i'_{\alpha},j'_{\alpha}), 
t'_{\alpha}), \cQ'_{\alpha}$ be the base change of 
$(
(U_{\alpha},\ol{U}_{\alpha},\cP_{\alpha},i_{\alpha},j_{\alpha}), 
t_{\alpha}), \cQ_{\alpha}$ by $\Spf O_{K'} \lra \Spf O_K$ respectively, 
let $\cP''_{\alpha} \subseteq 
\cP'_{\alpha}$ be a non-empty affine open formal subscheme such that 
$\cQ''_{\alpha} := \cQ'_{\alpha} \cap \cP''_{\alpha}$ is irreducible and let 
us put 
$$(
(U''_{\alpha},\ol{U}''_{\alpha},\cP''_{\alpha},i''_{\alpha},j''_{\alpha}), 
t''_{\alpha}) := 
 (
(U'_{\alpha} \cap \cP''_{\alpha}, \ol{U}'_{\alpha} \cap \cP''_{\alpha}, 
\cP''_{\alpha},i'_{\alpha}|_{\ol{U}''_{\alpha}},
j'_{\alpha}|_{U''_{\alpha}}), t'_{\alpha}|_{\cP''_{\alpha}}). $$
Let $E'_{\cE, \alpha}$ be the restriction of $E_{\cE,\alpha}$ to 
$\cQ''_{\alpha,K} \times A^1_K[\lam,1)$ (which is nothing but the 
$\nabla$-module induced by $\pi^*\cE$) and 
let $a'_{\alpha}$ be the unique index satisfying 
$\ol{U}''_{\alpha} \cap Z' \subseteq Z'_{a'_{\alpha}}$. 
(Then we have $\pi (a'_{\alpha}) = a_{\alpha}$.) 
Let us take an injection $\Gamma(\cQ''_{\alpha,K},\cO) \hra L$ into 
a field complete with respect to a norm which restricts to the 
supremum norm on $\cQ''_{\alpha,K}$. Then the composite 
$$ 
\Gamma(\cQ_{\alpha,K},\cO) \lra \Gamma(\cQ'_{\alpha,K},\cO) \lra 
\Gamma(\cQ''_{\alpha,K},\cO) \hra L $$
is also an injection such that the norm on $L$ restricts to 
the supremum norm. Now let us assume that $\pi^*\cE$ is 
$\Sigma'$-unipotent. Then, by Proposition \ref{wdsigmon}, 
$E'_{\cE,\alpha}$ is $\Sigma_{\pi(a'_{\alpha})}=
\Sigma_{a_{\alpha}}$-unipotent. So 
the restriction of $E'_{\cE,\alpha}$ to $A^1_L[\lam,1)$ is 
$\Sigma_{a_{\alpha}}$-unipotent. In other words, 
the restriction of $E_{\cE,\alpha}$ to $A^1_L[\lam,1)$ is 
$\Sigma_{a_{\alpha}}$-unipotent. Then, by Proposition 
\ref{generization}, $E_{\cE,\alpha}$ is 
$\Sigma_{a_{\alpha}}$-unipotent on $\cQ_{\alpha,K} \times A^1_K[\lam',1)$ 
for $\lam' \in (\lam, 1) \cap \Gamma^*$. Since this is true for any $\alpha$, 
we can conclude that $\cE$ has $\Sigma$-unipotent monodromy, as desired. 
So we are done. 
\end{proof}

\end{document}